\documentclass[12pt,twoside]{amsart}
\usepackage{amssymb,amsmath,amscd,enumerate,verbatim,xcolor,mathtools,fullpage}
\usepackage[T1,T5]{fontenc}
\usepackage[top=3.5cm, bottom=2.5cm, left=3.0cm, right=3cm]{geometry}
\usepackage{color, graphicx, wrapfig}
\usepackage{indentfirst}
\usepackage{tikz}
\usepackage{array}
\usetikzlibrary{calc}
\usepackage{mathrsfs}
\usepackage{graphicx}
\usepackage[all]{xy}
\usepackage{wrapfig}
\usetikzlibrary{arrows}
\usepackage{tikz-cd}
\usepackage{cleveref}

\newcommand{\kk}{\Bbbk}

\numberwithin{equation}{section}

\newtheorem{thm}{ Theorem}[section]
\newtheorem{lem}[thm]{ Lemma}

\newtheorem{prop}[thm]{ Proposition}

\theoremstyle{definition}
\newtheorem{defn}[thm]{Definition}

\newcommand{\br}[1]{\left (#1\right)}
\newcommand{\brc}[1]{\left[ #1\right]}
\newcommand{\brb}[1]{\left\{#1\right\}}
\newcommand{\abs}[1]{\left| #1\right|}

\newcommand{\ceil}[1]{ \left \lceil #1 \right \rceil}

\begin{document}
\begin{center}
{\bf  THE WEAK LEFSCHETZ PROPERTIES OF ARTINIAN MONOMIAL ALGEBRAS ASSOCIATED TO CERTAIN TADPOLE GRAPHS}
\vspace{0.5cm}\\
{\small PHAN MINH HUNG, NGUYEN DUY PHUOC$^\ast$, AND TRAN NGUYEN THANH SON}
\vspace{0.5cm}\\
University of Education, Hue University, 34 Le Loi St., Hue City, Viet Nam.\\
Corresponding author$^\ast$: ndphuoc@dhsphue.edu.vn
\end{center}
\vspace{0.5cm}
\begin{center}
\begin{minipage}{15cm}
{\small {\bf Abstract:} Given a simple graph $G$, the artinian monomial algebra associated to $G$, denoted by $A(G)$, is defined by the edge ideal of $G$ and the squares of the variables. In this article, we classify some tadpole graphs $G$ for which $A(G)$ has or fails the weak Lefschetz property.
		
}
	\end{minipage} 
\end{center}

\section{Introduction}
Let us consider the standard graded artinian algebra $A=\bigoplus_{i=0}^s\brc{A}_i= R/I$, where $R= \Bbbk\brc{x_1,x_2,...,x_n}$ is a polynomial ring over a field $\Bbbk$, all $x_i$'s have degree $1$, and $I\subset R$ is an artinian homogenous ideal of $R$.

\begin{defn}
    We say that $A$ has the \textit{weak Lefschetz property} (WLP for short) if there exists a linear form $\ell\in \brc{A}_1$ such that the multiplication map
    \begin{align*}
        \times \ell :\brc{A}_j\longrightarrow \brc{A}_{j+1}
    \end{align*}
    has maximal rank, i.e., it is injecive or surjective, for all $j=0,1,\ldots,s-1$. In this case the linear form $\ell$ is called a \textit{Lefschetz element} of $A$. 
\end{defn} 
The Lefschetz property is an algebrization of the Hard Lefschetz theorem, which is one of the most important theorems in algebraic geometry. Studying the weak Lefschetz property gives us many applications and information in other areas, such as poset theory, Schur-Weyl duality (see, for instance, \cite{HMMNWW2013}). 

The case of artinian $\kk$-algebras defined by monomial ideals, while being rather accessible, is far from simple and the literature concerning their Lefschetz properties is quite extensive; see, for instance, \cite{AB2020,DaoNair2022, MMN2011,MNS2020}  and the references therein.  In this work, we focus on a special class of artinian algebras defined by quadratic monomials which was defined and studied in \cite{NT2024, QHT2020}. Let $G= (V,E)$ be a simple graph where $V$ is a set of elements called \textit{vertices}, and $E$ a set of elements called \textit{edges} which are unorderd pairs of vertices from $V$. Suppose that $V= \brb{1,2,\ldots,n}$ and let $R= \Bbbk\brc{x_1,x_2,\ldots,x_n}$ be a standard graded polynomial ring over a field $\Bbbk$. The \textit{edge ideal} of $G$ is the ideal $I(G)= \br{\brb{x_ix_j\mid \brb{i,j}\in E}}\subset R$.  The \textit{artinian monomial algebra associated to $G$} is defined by
\begin{align*}
    A(G) = \frac{R}{\br{x_1^2,x_2^2,\ldots,x_n^2}+I(G)}.
\end{align*}
 
We are interested in studying the WLP of $A(T_{m,n})$ for certain tadpole graphs $T_{m,n}$.
Recall that the tadpole graph, denoted by $T_{m,n}$,  is the graph obtained by joining a cycle $C_m$ to a path $P_n$ with a bridge   (\Cref{fig1}).

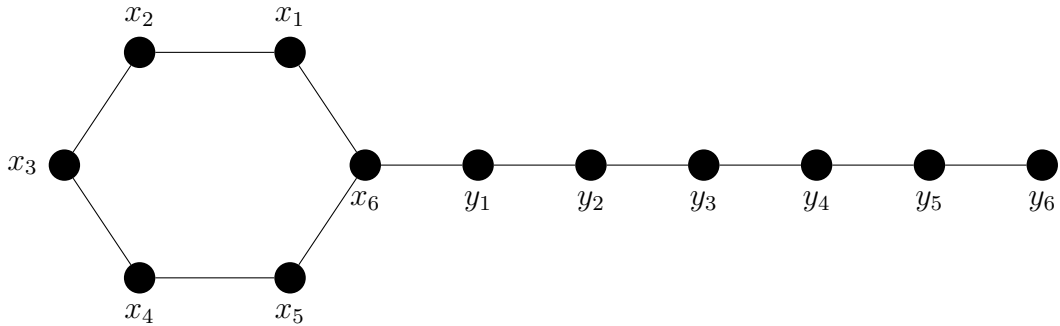
\begin{figure}[!ht]
	\begin{tikzpicture}[
		every edge/.style = {draw=black,very thick},
		vrtx/.style args = {#1/#2}{%
			circle, draw, thick, fill=black,
			minimum size=1mm, label=#1:#2}
		]
		\node (n1) [vrtx=above/$x_2$]  at (-1,0) {};
		\node (n2) [vrtx=above/$x_1$]at (1,0)  {};
		\node (n3) [vrtx=below/$x_6$]at (2,-1.5)  {};
		\node (n4) [vrtx=below/$x_5$]at (1,-3) {};
		\node (n5) [vrtx=below/$x_4$]at (-1,-3)  {};
		\node (n6) [vrtx=left/$x_3$]at (-2,-1.5)  {};
		\node (n7) [vrtx=below/$y_1$]at (3.5,-1.5)  {};
		\node (n8) [vrtx=below/$y_2$]at (5,-1.5)  {};
		\node (n9) [vrtx=below/$y_3$]at (6.5,-1.5)  {};
		\node (n10) [vrtx=below/$y_4$]at (8,-1.5)  {};
		\node (n11) [vrtx=below/$y_5$]at (9.5,-1.5)  {};
		\node (n12) [vrtx=below/$y_6$]at (11,-1.5)  {};
		\foreach \from/\to in {n1/n2,n2/n3,n3/n4,n4/n5,n5/n6,n6/n1, n3/n7, n8/n7, n8/n9, n9/n10, n10/n11, n11/n12}		
		\draw (\from) -- (\to);	
	\end{tikzpicture}
	\caption{Tadpole $T_{6,6}$}
	\label{fig1}
\end{figure}
 Note that the cases where $m=3$ or $n=1$ were studied in \cite{NT2024}. Our main goal in this note is to investigate the WLP of $A(T_{m,n})$ for $m\in\brb{4,5}$ or $n\in \brb{2,3}$. Our main results are the following.
\begin{thm}[Theorem \ref{wlp_Tm2}, \ref{wlp_Tm3}, \ref{wlp_T4n} and \ref{wlp_T5n}] \label{thm1.2}
Assume that $\Bbbk$ is of characteristic zero. Then
\begin{enumerate}[\quad \rm (i)]
\item  $A\br{T_{m,2}}$ has the WLP if and only if $m\in \brb{4,5,7,8,11}$.
\item $A\br{T_{m,3}}$ has the WLP if and only if $m\in \brb{3,4,5,6,7,8,10,11,14}$.
\item $A\br{T_{4,n}}$ has the WLP if and only if $n\in \brb{1,2,\ldots,7,9,10,13}$.
\item $A\br{T_{5,n}}$ has the WLP if and only if $n\in \brb{1,2,3,5,6,9}$. 
\end{enumerate}
\end{thm}
The proof combines {\tt Macaulay2} \cite{Macaulay2} computations with inductive arguments based on the unimodality of the independence polynomials of the relevant graphs.

Our paper is structured as follows. In the next section we recall relevant terminology and results on artinian algebras, Lefschetz properties, and graph theory. In Section~3, we investigate the unimodality and the mode of the independence polynomials of certain tadpole graphs. These results are useful to prove \Cref{thm1.2} in Section~4.

\section{Preliminaries}
In this section we recall some standard terminology and notations from commutative algebra and combinatorial commutative algebra, as well as some results needed later on.
\subsection{The weak Lefschetz property}
In this paper we consider artinian algebras defined by monomial ideals, and in this case it suffices to choose the Lefschetz element to be the sum of the variables.
\begin{prop}{\rm \cite[Proposition 2.2]{MMN2011}}
    Let $I\subset R=\Bbbk\brc{x_1,x_2,...,x_n}$ be an artinian monomial ideal. Then $A= R/I$ has the WLP if and only if $\ell = x_1+x_2+\cdots+x_n$ is a Lefschetz element for $A$.
\end{prop}

A necessary condition for the WLP of an artinian algebra $A$ is the unimodality of the Hilbert series of $A$. 

\begin{defn}
 Let $A=\bigoplus_{j\geq 0} [A]_j$ be a standard graded $\kk$-algebra. The {\it Hilbert series} of $A$ is the power series $\sum\dim_\kk [A]_i t^i$ and is denoted by $HS(A,t)$. The {\it Hilbert function} of $A$ is the function  $h_A: \mathbb{N}\longrightarrow \mathbb{N}$ defined by $h_A(j)=\dim_\kk [A]_j$.\\
\end{defn}
If $A$ is an artinian graded algebra, then $[A]_i=0$ for $i\gg 0.$  Denote
$$D=\max\{i\mid [A]_i\neq 0\},$$
the \emph{socle degree} of $A$. In this case, the Hilbert series of $A$ is a polynomial
$$HS(A,t)=1+h_1t+\cdots+ h_Dt^D,$$
where $h_i=\dim_\Bbbk [A]_i>0$. By definition, the degree of the Hilbert series for an artinian graded algebra $A$ is equal to its socle degree $D.$ 

\begin{defn}
A polynomial $\sum_{k=0}^na_kt^k \in \mathbb{R}[t]$ with non-negative coefficients is called {\it unimodal} if there is some $m$, such that
$$a_0\leq a_1\leq \cdots \leq a_{m-1}\leq a_m \geq a_{m+1}\geq \cdots \geq a_n.$$ 
Set $a_{-1}=0$. The {\it mode} of the unimodal polynomial $\sum_{k=0}^na_kt^k$ is defined to be the unique integer $i$ between $0$ and $n$ such that
$$
a_{i-1}< a_i \geq a_{i+1}\geq \cdots \geq a_n.
$$
\end{defn}

\begin{prop}{\rm \cite[Proposition 3.2]{HMMNWW2013}}\label{unimodalityofHilbertseries}
If $A$ has the WLP then the Hilbert series of $A$ is unimodal.
\end{prop}
Finally, to study the failure of the WLP of tensor products of $\kk$-algebras, the following simple lemma turns out to be quite useful. 
\begin{lem}{\rm \cite[Lemma 7.8]{BMMNZ12}}\label{tensor}
Let $A= A'\otimes_\Bbbk A''$ be the tensor product of two graded artinian $\Bbbk-$algebras $A'$ and $A''$. Let $\ell'\in A'$ and $\ell''\in A''$ be linear elements, and set $\ell= \ell' +\ell''= \ell'\otimes 1+1\otimes \ell''\in A$. Then 
\begin{enumerate}[\quad \rm (i)]
\item  If the multiplication maps $\times\ell':\brc{A'}_i\longrightarrow\brc{A'}_{i+1}$ and $\times\ell'':\brc{A''}_{j}\longrightarrow\brc{A''}_{j+1}$ are both not surjective, then neither is the map $\times\ell:\brc{A}_{i+j+1}\longrightarrow\brc{A}_{i+j+2}$.
\item If the multiplication maps $\times\ell':\brc{A'}_i\longrightarrow\brc{A'}_{i+1}$ and $\times\ell'':\brc{A''}_j\longrightarrow\brc{A''}_{j+1}$ are both not injective, then neither is the map $\times\ell: \brc{A}_{i+j}\longrightarrow\brc{A}_{i+j+1}$. 
\end{enumerate}
\end{lem}

\subsection{Graph theory} 
From now on, by a graph we mean a simple graph $G=(V,E)$ with the vertex set $V=V(G)$ and the edge set $E=E(G)$. 
We start by recalling some basic definitions. 
\begin{defn}
The {\it disjoint union} of the graphs $G_1$ and $G_2$ is a graph $G=G_1\cup G_2$ having as vertex set the disjoint union of $V(G_1)$ and $V(G_2)$, and as edge set the disjoint union of $E(G_1)$ and $E(G_2)$. In particular, $\cup_{m}G$ denotes the disjoint union of $m>1$ copies of the graph $G$.
\end{defn}
\begin{defn}
Let $G=(V,E)$ be a graph.
\begin{enumerate}[\quad \rm (i)]
\item  A subset $X$ of $V$ is called an {\it independent set} of $G$ if for any $u,v\in X,\ \{u,v\}\notin E$, i.e., the vertices in $X$ are pairwise non-adjacent. If an independent set $X$ has $k$ elements, then we say that $X$ is an {\it independent set of size $k$} or a $k$-independent set of $G$.
\item The {\it independence number} of a graph $G$ is the largest cardinality of an independent set of $G$. We denote this value by $\alpha(G)$.
\end{enumerate}
\end{defn}
\begin{defn}
The {\it independence polynomial} of a graph $G$ is a polynomial in one variable $t$ whose coefficient of $t^k$ is given by the number of independent sets of size $k$ of $G$. We denote this polynomial by $I(G;t)$, i.e.,
$$I(G;t)=\sum_{k=0}^{\alpha(G)}s_k(G)t^k,$$
where $s_k(G)$ is the number of independent sets of size $k$ in $G$. Note that $s_0(G)=1$ since $\emptyset$ is an independent set of any graph $G$.
\end{defn}
The independence polynomial of a graph was defined by Gutman and Harary in \cite{GH83} as a generalization of the matching polynomial of a graph. 
For a vertex $v\in V$, its \emph{open neighborhood} $N(v)$ is the set of vertices $u\neq v$ that are adjacent to $v$, and its \emph{closed neighborhood} is $N[v]= N(v) \cup \{v\}$. For a subset $U\subset V$, let $G\setminus U$ denote the graph obtained from $G$ by deleting all vertices in $U$ and all edges adjacent to those vertices. In particular, a vertex $v\in V$, we simply write $G\setminus v$ instead of  $G\setminus \{v\}$. The following equalities are very useful to compute the independent polynomials of various families of graphs.
\begin{prop}{\rm \cite[Theorem 2.3 and Corollary 3.3]{HL94}}\label{comp_inde_graph}
Let $G_1,G_2, G$ be the graphs. Assume that $G=(V,E)$ and $ v\in V$. Then the following equalities hold:
\begin{enumerate}
\item [\rm (i)] $I(G;t)=I(G\setminus v;t)+t\cdot I(G\setminus N[v];t)$;
\item [\rm (ii)] $I(G_1\cup G_2;t)=I(G_1;t)I(G_2;t)$.
\end{enumerate}
\end{prop}

\subsection{Artinian monomial algebras associated to graphs}
A connection between combinatorial information of a graph and the artinian monomial algebra associated to it is given as follows.

\begin{prop}{\rm \cite[Proposition 2.10]{NT2024}}
The Hilbert series of $A(G)$ is equal to the independent polynomial of $G$.
\end{prop}
Therefore, the WLP of $A(G)$ has strong consequences on the unimodality of the independence polynomial of $G$ by \Cref{unimodalityofHilbertseries}.

We close this section by recalling some results regarding  paths $P_n$, cycles $C_n$, and Pan graphs $\mbox{Pan}_n$ (i.e., tadpole graph $T_{n,1}$). In \cite{NT2024}, the independence polynomials of these graphs are unimodal. Denote by $\lambda_n,\rho_n$ and $\zeta_n$ the mode of $I\br{P_n;t},I\br{C_n;t}$ and $I\br{\mbox{Pan}_n;t}$, respectively. 

\begin{prop}{\rm \cite[Lemmas 3.2, 3.4, 3.5 and 3.6]{NT2024}}\label{compare_modes}
The following inequalities hold:
\begin{enumerate}[\quad \rm (i)]
\item  For all $n\geq 1$, there are inequalities $\lambda_{n+1}\geq \lambda_n$, $\lambda_{n+3}-1\leq \lambda_n\leq \lambda_{n+4}-1$.
\item For all $n\geq 5$, there are inequalities $\lambda_{n-1}\leq \rho_n\leq \lambda_{n-4}+1\leq \lambda_n$.
\item For all $n\geq 5$, there are inequalities $\rho_n\leq \lambda_n\leq \zeta_n\leq \rho_n+1\leq\lambda_n+1$. 
\end{enumerate}
\end{prop}
\begin{thm}{\rm \cite[Theorem 4.2 and Proposition 4.3]{NT2024}}\label{thm_Pn}
 Assume that $\Bbbk$ is of characteristic zero. For an integer $n\ge 1$, $A\br{P_n}$ has the WLP if and only if $n\in \brb{1,2,\ldots,7,9,10,13}$. In particular, one has
 \begin{enumerate}[\quad \rm (i)]
\item For all $n\geq 17$, $A\br{P_n}$ fails the surjectivity at degree $\lambda_n$.
\item If $n\geq 12$ is an integer such that $\lambda_n=\lambda_{n-1}+1$, then $A\br{P_n}$ fails the injectivity from degree $\lambda_n-1$ to $\lambda_n$.
 \end{enumerate}
\end{thm}

\begin{thm}{\rm \cite[Theorem 4.4]{NT2024}}\label{thm_Cn}
  Assume that $\Bbbk$ is of characteristic zero. For an integer $n\ge 3$, $A\br{C_n}$ has the WLP if and only if $n\in \brb{3,4,\ldots,11,13,14,17}$. In particular, for all $n\geq 21$, $A\br{C_n}$ fails the surjectivity at degree $\rho_n$.
\end{thm}

\section{Independence polynomial of some tadpole graphs}

In this paper, we will consider some tadpole graphs, that are $T_{m,2}$, $T_{m,3}$, $T_{4,n}$ and $T_{5,n}$. To study the unimodality of polynomials, the following result is useful. Note that given a polynomial $f(x)= \sum_{i=0}^na_ix^i$, we will regard $a_k=0$ for all $k>n$ or $k<0$.

\begin{lem}\label{unimodal_mode_unit}
Let $f$ and $g$ be two unimodal polynomials with real-nonnegative coefficients and modes $p,q$, respectively, such that $\abs{p-q}\leq 1$. Then $f+g$ is also unimodal whose mode belongs to $\brb{\min\brb{p,q},\min\brb{p,q}+1}$.
\end{lem}
\begin{proof}
Without loss of generality, assume that $p\leq q$. Assume that
\begin{align*}
&f(x) = a_0 + a_1x + a_2x^2 + \cdots + a_{n-1}x^{n-1}+ a_nx^n,\\
&g(x) = b_0 + b_1x + b_2x^2 + \cdots + b_{m-1}x^{m-1}+ b_mx^m.
\end{align*}
	
Then
\begin{align*}
&a_0 \leq a_1\leq \cdots \leq a_{p-1}< a_p\geq a_{p+1}\geq\cdots \geq a_n,\\
&b_0 \leq b_1\leq \cdots \leq b_{q-1}<b_q\geq b_{q+1}\geq \cdots \geq b_m.
\end{align*}
If $p=q$, then $f+g$ is unimodal with mode $p$. Now, if $q=p+1$, it is easy to see that
\begin{align*}
&a_{i-1}+b_{i-1}\leq a_i+b_i,\forall i=1,\ldots,p-1,\\
&a_i+b_i\geq a_{i+1}+b_{i+1},\forall i=p+1,\ldots,\max\brb{m,n}\\
&a_{p-1}+b_{p-1}<a_p+b_p.
\end{align*}
If $a_p+b_p\geq a_{p+1}+b_{p+1}$, then $f+g$ is unimodal with mode $p$. And if $a_p+b_p<a_{p+1}+b_{p+1}$, then $f+g$ is unimodal with mode $p+1$.
\end{proof}
\begin{figure}[!ht]
	\begin{tikzpicture}[
		every edge/.style = {draw=black,very thick},
		vrtx/.style args = {#1/#2}{%
			circle, draw, thick, fill=black,
			minimum size=1mm, label=#1:#2}
		]
		\node (n1) [vrtx=above/$x_2$]  at (-1,0) {};
		\node (n2) [vrtx=above/$x_1$]at (1,0)  {};
		\node (n3) [vrtx=below/$x_m$]at (2,-1.5)  {};
		\node (n4) [vrtx=below/$x_{m-1}$]at (1,-3) {};
		\node (n5) [vrtx=below/$$]at (-1,-3)  {};
		\node (n6) [vrtx=left/$$]at (-2,-1.5)  {};
		\node (n7) [vrtx=below/$y_1$]at (3.5,-1.5)  {};
		\node (n8) [vrtx=below/$y_2$]at (5,-1.5)  {};
		\foreach \from/\to in {n1/n2,n2/n3,n3/n4,n4/n5,n5/n6,n6/n1, n3/n7, n8/n7}		
		\draw (\from) -- (\to);	
	\end{tikzpicture}
    \caption{The tadpole $T_{m,2}$}
    \label{fig2}
\end{figure}
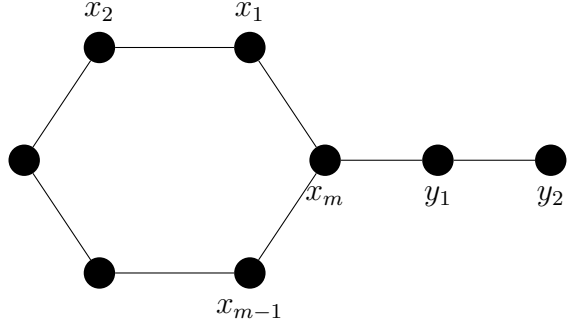
Recall that $\rho_m$ is the mode of the independence polynomial of  $C_m$. 
\begin{prop}\label{modeT_{m,n}}
$I\br{T_{m,2};t}$ is unimodal with the mode belongs to $\brb{\rho_m,\rho_m+1}$, for all $m\ge 5$.
\end{prop}

\begin{proof}
Applying  \Cref{comp_inde_graph}(i) for the vertex numbered $y_2$ (\Cref{fig2})
\begin{equation*}
I\br{T_{m,2};t}=I\br{T_{m,2}\setminus y_2;t}+tI\br{T_{m,2}\setminus N\brc{y_2};t}=I\br{\mbox{Pan}_m;t}+tI\br{C_m;t}
\end{equation*}
    
By \Cref{compare_modes}(iii), we have two following cases:\\
\underline{\textbf{Case 1.} $\zeta_m=\rho_m$.} It is easy to see that the mode of $tI\br{C_m;t}$ is $\rho_m+1$. Hence, by \Cref{unimodal_mode_unit}, we have that $I\br{T_{m,2};t}$ is unimodal whose mode belongs to $\brb{\rho_m,\rho_m+1}$.\\
\underline{\textbf{Case 2.}  $\zeta_m=\rho_m+1$.} Because the mode of $tI\br{C_m;t}$ is $\rho_m+1$, $I\br{T_{m,2};t}$ is unimodal whose mode is $\rho_m+1$. 

    We conclude that $I\br{T_{m,2};t}$ is unimodal with the mode belongs to $\brb{\rho_m,\rho_m+1}$.
\end{proof}

\begin{prop}
$I\br{T_{m,3};t}$ is unimodal with the mode belongs to $\brb{\rho_m,\rho_m+1,\rho_m+2}$, for all $m\ge 5$.
\end{prop}
\begin{proof}
 \begin{figure}[!ht]
	\begin{tikzpicture}[
		every edge/.style = {draw=black,very thick},
		vrtx/.style args = {#1/#2}{%
			circle, draw, thick, fill=black,
			minimum size=1mm, label=#1:#2}
		]
		\node (n1) [vrtx=above/$x_2$]  at (-1,0) {};
		\node (n2) [vrtx=above/$x_1$]at (1,0)  {};
		\node (n3) [vrtx=below/$x_m$]at (2,-1.5)  {};
		\node (n4) [vrtx=below/$x_{m-1}$]at (1,-3) {};
		\node (n5) [vrtx=below/$$]at (-1,-3)  {};
		\node (n6) [vrtx=left/$$]at (-2,-1.5)  {};
		\node (n7) [vrtx=below/$y_1$]at (3.5,-1.5)  {};
		\node (n8) [vrtx=below/$y_2$]at (5,-1.5)  {};
        \node (n9) [vrtx=below/$y_3$]at (6.5,-1.5)  {};
		\foreach \from/\to in {n1/n2,n2/n3,n3/n4,n4/n5,n5/n6,n6/n1, n3/n7, n8/n7,n8/n9}		
		\draw (\from) -- (\to);	
	\end{tikzpicture}
    \caption{The tadpole $T_{m,3}$}
    \label{fig3}
\end{figure}
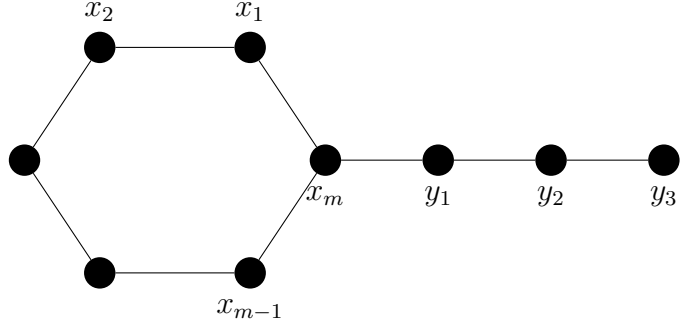
Applying  \Cref{comp_inde_graph}(i) for the vertex numbered $y_2$ (\Cref{fig3})
    \begin{equation*}
        I\br{T_{m,3};t}=I\br{T_{m,3}\setminus y_2;t}+tI\br{T_{m,3}\setminus N\brc{y_2};t}=\br{1+t}I\br{\mbox{Pan}_m;t}+tI\br{C_m;t}
    \end{equation*}
    
By \Cref{compare_modes}, we have two following cases:\\
    \underline{\textbf{Case 1.} $\zeta_m=\rho_m$.} Applying \Cref{unimodal_mode_unit}, $\br{1+t}I\br{\mbox{Pan}_m;t}= I\br{\mbox{Pan}_m;t}+tI\br{\mbox{Pan}_m;t}$ is unimodal whose mode is in  $\brb{\rho_m,\rho_m+1}$. On the other hand, the mode of $tI\br{C_m;t}$ is $\rho_m+1$. Hence, applying \Cref{unimodal_mode_unit} again, one has that $I\br{T_{m,3};t}$ is also unimodal whose mode belongs to $\brb{\rho_m,\rho_m+1}$.\\
    \underline{\textbf{Case 2.} $\zeta_m=\rho_m+1$.} Applying \Cref{unimodal_mode_unit}, $\br{1+t}I\br{\mbox{Pan}_m;t}$ is unimodal whose mode is in  $\brb{\rho_m+1,\rho_m+2}$. On the other hand, the mode of $tI\br{C_m;t}$ is $\rho_m+1$. Hence, applying \Cref{unimodal_mode_unit} again, we have that $I\br{T_{m,3};t}$ is also unimodal whose mode belongs to $\brb{\rho_m+1,\rho_m+2}$.

     We conclude that $I\br{T_{m,3};t}$ is unimodal with the mode belongs to $\brb{\rho_m,\rho_m+1,\rho_m+2}$.
\end{proof}

\begin{prop}\label{T_4,n}
Recall the mode $\lambda_n$ of the independence polynomial of $P_n$. Then  $I\br{T_{4,n};t}$ is unimodal with the mode belongs to $\brb{\lambda_{n+2},\lambda_{n+2}+1}$, for all $n\ge 5$.
\end{prop}
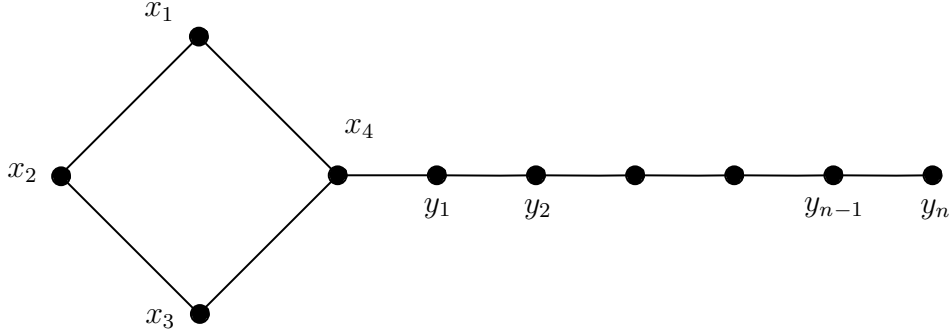
\begin{figure}[!ht]
    \centering
\tikzset{every picture/.style={line width=0.75pt}} 
\begin{tikzpicture}[x=0.75pt,y=0.75pt,yscale=-1,xscale=1]
\draw  [line width=3] [line join = round][line cap = round] (310,120.6) .. controls (310,120.6) and (310,120.6) .. (310,120.6) ;
\draw    (360,120) -- (391.76,120.18) -- (410,120) ;
\draw  [line width=3] [line join = round][line cap = round] (410,120) .. controls (410,120) and (410,120) .. (410,120) ;
\draw  [line width=3] [line join = round][line cap = round] (170.77,120.84) .. controls (170.77,120.84) and (170.77,120.84) .. (170.77,120.84) ;
\draw    (310,120) -- (360,120) ;
\draw  [line width=3] [line join = round][line cap = round] (360,120) .. controls (360,120) and (360,120) .. (360,120) ;
\draw    (460,120) -- (491.76,120.18) -- (510,120) ;
\draw  [line width=3] [line join = round][line cap = round] (610,120) .. controls (610,120) and (610,120) .. (610,120) ;
\draw    (410,120) -- (441.76,120.18) -- (460,120) ;
\draw  [line width=3] [line join = round][line cap = round] (560,120) .. controls (560,120) and (560,120) .. (560,120) ;
\draw  [line width=3] [line join = round][line cap = round] (460,120) .. controls (460,120) and (460,120) .. (460,120) ;
\draw  [line width=3] [line join = round][line cap = round] (511,120) .. controls (511,120) and (511,120) .. (511,120) ;
\draw    (510,120) -- (541.76,120.18) -- (560,120) ;
\draw    (560,120) -- (591.76,120.18) -- (610,120) ;
\draw   (364.5,120) .. controls (364.5,117.51) and (362.49,115.5) .. (360,115.5) .. controls (357.51,115.5) and (355.5,117.51) .. (355.5,120) .. controls (355.5,122.49) and (357.51,124.5) .. (360,124.5) .. controls (362.49,124.5) and (364.5,122.49) .. (364.5,120) -- cycle ;
\draw  [line width=3] [line join = round][line cap = round] (359,119.2) .. controls (359.67,119.2) and (360.53,119.67) .. (361,119.2) .. controls (361.24,118.96) and (361.3,118.35) .. (361,118.2) .. controls (360.16,117.78) and (358.33,119.53) .. (359,120.2) .. controls (360.33,121.53) and (362.33,119.53) .. (361,118.2) .. controls (357.78,114.98) and (359.97,122.17) .. (358,120.2) .. controls (356.22,118.42) and (361.78,119.98) .. (360,118.2) .. controls (359.03,117.23) and (358.39,120.97) .. (359,122.2) .. controls (359.3,122.8) and (360.53,122.67) .. (361,122.2) .. controls (361.47,121.73) and (361.67,120.2) .. (361,120.2) ;
\draw  [line width=3] [line join = round][line cap = round] (363,120.2) .. controls (362.67,120.2) and (362.33,120.2) .. (362,120.2) ;
\draw  [line width=3] [line join = round][line cap = round] (362.28,118.34) .. controls (362.28,118.34) and (362.28,118.34) .. (362.28,118.34) ;
\draw  [line width=3] [line join = round][line cap = round] (358.28,118.34) .. controls (358.28,118.34) and (358.28,118.34) .. (358.28,118.34) ;
\draw  [line width=3] [line join = round][line cap = round] (358.28,121.34) .. controls (358.28,121.34) and (358.28,121.34) .. (358.28,121.34) ;
\draw  [line width=3] [line join = round][line cap = round] (359.28,119.34) .. controls (359.28,119.34) and (359.28,119.34) .. (359.28,119.34) ;
\draw   (414.5,120) .. controls (414.5,117.51) and (412.49,115.5) .. (410,115.5) .. controls (407.51,115.5) and (405.5,117.51) .. (405.5,120) .. controls (405.5,122.49) and (407.51,124.5) .. (410,124.5) .. controls (412.49,124.5) and (414.5,122.49) .. (414.5,120) -- cycle ;
\draw  [line width=3] [line join = round][line cap = round] (409,119.2) .. controls (409.67,119.2) and (410.53,119.67) .. (411,119.2) .. controls (411.24,118.96) and (411.3,118.35) .. (411,118.2) .. controls (410.16,117.78) and (408.33,119.53) .. (409,120.2) .. controls (410.33,121.53) and (412.33,119.53) .. (411,118.2) .. controls (407.78,114.98) and (409.97,122.17) .. (408,120.2) .. controls (406.22,118.42) and (411.78,119.98) .. (410,118.2) .. controls (409.03,117.23) and (408.39,120.97) .. (409,122.2) .. controls (409.3,122.8) and (410.53,122.67) .. (411,122.2) .. controls (411.47,121.73) and (411.67,120.2) .. (411,120.2) ;
\draw  [line width=3] [line join = round][line cap = round] (413,120.2) .. controls (412.67,120.2) and (412.33,120.2) .. (412,120.2) ;
\draw  [line width=3] [line join = round][line cap = round] (412.28,118.34) .. controls (412.28,118.34) and (412.28,118.34) .. (412.28,118.34) ;
\draw  [line width=3] [line join = round][line cap = round] (408.28,118.34) .. controls (408.28,118.34) and (408.28,118.34) .. (408.28,118.34) ;
\draw  [line width=3] [line join = round][line cap = round] (408.28,121.34) .. controls (408.28,121.34) and (408.28,121.34) .. (408.28,121.34) ;
\draw  [line width=3] [line join = round][line cap = round] (409.28,119.34) .. controls (409.28,119.34) and (409.28,119.34) .. (409.28,119.34) ;
\draw   (464.5,120) .. controls (464.5,117.51) and (462.49,115.5) .. (460,115.5) .. controls (457.51,115.5) and (455.5,117.51) .. (455.5,120) .. controls (455.5,122.49) and (457.51,124.5) .. (460,124.5) .. controls (462.49,124.5) and (464.5,122.49) .. (464.5,120) -- cycle ;
\draw  [line width=3] [line join = round][line cap = round] (459,119.2) .. controls (459.67,119.2) and (460.53,119.67) .. (461,119.2) .. controls (461.24,118.96) and (461.3,118.35) .. (461,118.2) .. controls (460.16,117.78) and (458.33,119.53) .. (459,120.2) .. controls (460.33,121.53) and (462.33,119.53) .. (461,118.2) .. controls (457.78,114.98) and (459.97,122.17) .. (458,120.2) .. controls (456.22,118.42) and (461.78,119.98) .. (460,118.2) .. controls (459.03,117.23) and (458.39,120.97) .. (459,122.2) .. controls (459.3,122.8) and (460.53,122.67) .. (461,122.2) .. controls (461.47,121.73) and (461.67,120.2) .. (461,120.2) ;
\draw  [line width=3] [line join = round][line cap = round] (463,120.2) .. controls (462.67,120.2) and (462.33,120.2) .. (462,120.2) ;
\draw  [line width=3] [line join = round][line cap = round] (462.28,118.34) .. controls (462.28,118.34) and (462.28,118.34) .. (462.28,118.34) ;
\draw  [line width=3] [line join = round][line cap = round] (458.28,118.34) .. controls (458.28,118.34) and (458.28,118.34) .. (458.28,118.34) ;
\draw  [line width=3] [line join = round][line cap = round] (458.28,121.34) .. controls (458.28,121.34) and (458.28,121.34) .. (458.28,121.34) ;
\draw  [line width=3] [line join = round][line cap = round] (459.28,119.34) .. controls (459.28,119.34) and (459.28,119.34) .. (459.28,119.34) ;
\draw  [line width=3] [line join = round][line cap = round] (610,120) .. controls (610,120) and (610,120) .. (610,120) ;
\draw   (614.5,120) .. controls (614.5,117.51) and (612.49,115.5) .. (610,115.5) .. controls (607.51,115.5) and (605.5,117.51) .. (605.5,120) .. controls (605.5,122.49) and (607.51,124.5) .. (610,124.5) .. controls (612.49,124.5) and (614.5,122.49) .. (614.5,120) -- cycle ;
\draw  [line width=3] [line join = round][line cap = round] (609,119.2) .. controls (609.67,119.2) and (610.53,119.67) .. (611,119.2) .. controls (611.24,118.96) and (611.3,118.35) .. (611,118.2) .. controls (610.16,117.78) and (608.33,119.53) .. (609,120.2) .. controls (610.33,121.53) and (612.33,119.53) .. (611,118.2) .. controls (607.78,114.98) and (609.97,122.17) .. (608,120.2) .. controls (606.22,118.42) and (611.78,119.98) .. (610,118.2) .. controls (609.03,117.23) and (608.39,120.97) .. (609,122.2) .. controls (609.3,122.8) and (610.53,122.67) .. (611,122.2) .. controls (611.47,121.73) and (611.67,120.2) .. (611,120.2) ;
\draw  [line width=3] [line join = round][line cap = round] (613,120.2) .. controls (612.67,120.2) and (612.33,120.2) .. (612,120.2) ;
\draw  [line width=3] [line join = round][line cap = round] (612.28,118.34) .. controls (612.28,118.34) and (612.28,118.34) .. (612.28,118.34) ;
\draw  [line width=3] [line join = round][line cap = round] (608.28,118.34) .. controls (608.28,118.34) and (608.28,118.34) .. (608.28,118.34) ;
\draw  [line width=3] [line join = round][line cap = round] (608.28,121.34) .. controls (608.28,121.34) and (608.28,121.34) .. (608.28,121.34) ;
\draw  [line width=3] [line join = round][line cap = round] (609.28,119.34) .. controls (609.28,119.34) and (609.28,119.34) .. (609.28,119.34) ;
\draw  [line width=3] [line join = round][line cap = round] (510,120) .. controls (510,120) and (510,120) .. (510,120) ;
\draw   (514.5,120) .. controls (514.5,117.51) and (512.49,115.5) .. (510,115.5) .. controls (507.51,115.5) and (505.5,117.51) .. (505.5,120) .. controls (505.5,122.49) and (507.51,124.5) .. (510,124.5) .. controls (512.49,124.5) and (514.5,122.49) .. (514.5,120) -- cycle ;
\draw  [line width=3] [line join = round][line cap = round] (509,119.2) .. controls (509.67,119.2) and (510.53,119.67) .. (511,119.2) .. controls (511.24,118.96) and (511.3,118.35) .. (511,118.2) .. controls (510.16,117.78) and (508.33,119.53) .. (509,120.2) .. controls (510.33,121.53) and (512.33,119.53) .. (511,118.2) .. controls (507.78,114.98) and (509.97,122.17) .. (508,120.2) .. controls (506.22,118.42) and (511.78,119.98) .. (510,118.2) .. controls (509.03,117.23) and (508.39,120.97) .. (509,122.2) .. controls (509.3,122.8) and (510.53,122.67) .. (511,122.2) .. controls (511.47,121.73) and (511.67,120.2) .. (511,120.2) ;
\draw  [line width=3] [line join = round][line cap = round] (513,120.2) .. controls (512.67,120.2) and (512.33,120.2) .. (512,120.2) ;
\draw  [line width=3] [line join = round][line cap = round] (512.28,118.34) .. controls (512.28,118.34) and (512.28,118.34) .. (512.28,118.34) ;
\draw  [line width=3] [line join = round][line cap = round] (508.28,118.34) .. controls (508.28,118.34) and (508.28,118.34) .. (508.28,118.34) ;
\draw  [line width=3] [line join = round][line cap = round] (508.28,121.34) .. controls (508.28,121.34) and (508.28,121.34) .. (508.28,121.34) ;
\draw  [line width=3] [line join = round][line cap = round] (509.28,119.34) .. controls (509.28,119.34) and (509.28,119.34) .. (509.28,119.34) ;
\draw  [line width=3] [line join = round][line cap = round] (560,120) .. controls (560,120) and (560,120) .. (560,120) ;
\draw   (564.5,120) .. controls (564.5,117.51) and (562.49,115.5) .. (560,115.5) .. controls (557.51,115.5) and (555.5,117.51) .. (555.5,120) .. controls (555.5,122.49) and (557.51,124.5) .. (560,124.5) .. controls (562.49,124.5) and (564.5,122.49) .. (564.5,120) -- cycle ;
\draw  [line width=3] [line join = round][line cap = round] (559,119.2) .. controls (559.67,119.2) and (560.53,119.67) .. (561,119.2) .. controls (561.24,118.96) and (561.3,118.35) .. (561,118.2) .. controls (560.16,117.78) and (558.33,119.53) .. (559,120.2) .. controls (560.33,121.53) and (562.33,119.53) .. (561,118.2) .. controls (557.78,114.98) and (559.97,122.17) .. (558,120.2) .. controls (556.22,118.42) and (561.78,119.98) .. (560,118.2) .. controls (559.03,117.23) and (558.39,120.97) .. (559,122.2) .. controls (559.3,122.8) and (560.53,122.67) .. (561,122.2) .. controls (561.47,121.73) and (561.67,120.2) .. (561,120.2) ;
\draw  [line width=3] [line join = round][line cap = round] (563,120.2) .. controls (562.67,120.2) and (562.33,120.2) .. (562,120.2) ;
\draw  [line width=3] [line join = round][line cap = round] (562.28,118.34) .. controls (562.28,118.34) and (562.28,118.34) .. (562.28,118.34) ;
\draw  [line width=3] [line join = round][line cap = round] (558.28,118.34) .. controls (558.28,118.34) and (558.28,118.34) .. (558.28,118.34) ;
\draw  [line width=3] [line join = round][line cap = round] (558.28,121.34) .. controls (558.28,121.34) and (558.28,121.34) .. (558.28,121.34) ;
\draw  [line width=3] [line join = round][line cap = round] (559.28,119.34) .. controls (559.28,119.34) and (559.28,119.34) .. (559.28,119.34) ;
\draw  [line width=3] [line join = round][line cap = round] (310,120) .. controls (310,120) and (310,120) .. (310,120) ;
\draw   (314.5,120) .. controls (314.5,117.51) and (312.49,115.5) .. (310,115.5) .. controls (307.51,115.5) and (305.5,117.51) .. (305.5,120) .. controls (305.5,122.49) and (307.51,124.5) .. (310,124.5) .. controls (312.49,124.5) and (314.5,122.49) .. (314.5,120) -- cycle ;
\draw  [line width=3] [line join = round][line cap = round] (309,119.2) .. controls (309.67,119.2) and (310.53,119.67) .. (311,119.2) .. controls (311.24,118.96) and (311.3,118.35) .. (311,118.2) .. controls (310.16,117.78) and (308.33,119.53) .. (309,120.2) .. controls (310.33,121.53) and (312.33,119.53) .. (311,118.2) .. controls (307.78,114.98) and (309.97,122.17) .. (308,120.2) .. controls (306.22,118.42) and (311.78,119.98) .. (310,118.2) .. controls (309.03,117.23) and (308.39,120.97) .. (309,122.2) .. controls (309.3,122.8) and (310.53,122.67) .. (311,122.2) .. controls (311.47,121.73) and (311.67,120.2) .. (311,120.2) ;
\draw  [line width=3] [line join = round][line cap = round] (313,120.2) .. controls (312.67,120.2) and (312.33,120.2) .. (312,120.2) ;
\draw  [line width=3] [line join = round][line cap = round] (312.27,118.34) .. controls (312.27,118.34) and (312.27,118.34) .. (312.27,118.34) ;
\draw  [line width=3] [line join = round][line cap = round] (308.27,118.34) .. controls (308.27,118.34) and (308.27,118.34) .. (308.27,118.34) ;
\draw  [line width=3] [line join = round][line cap = round] (308.27,121.34) .. controls (308.27,121.34) and (308.27,121.34) .. (308.27,121.34) ;
\draw  [line width=3] [line join = round][line cap = round] (309.27,119.34) .. controls (309.27,119.34) and (309.27,119.34) .. (309.27,119.34) ;
\draw  [line width=3] [line join = round][line cap = round] (170.5,120.5) .. controls (170.5,120.5) and (170.5,120.5) .. (170.5,120.5) ;
\draw   (175,120.5) .. controls (175,118.01) and (172.99,116) .. (170.5,116) .. controls (168.01,116) and (166,118.01) .. (166,120.5) .. controls (166,122.99) and (168.01,125) .. (170.5,125) .. controls (172.99,125) and (175,122.99) .. (175,120.5) -- cycle ;
\draw  [line width=3] [line join = round][line cap = round] (169.5,119.7) .. controls (170.17,119.7) and (171.03,120.17) .. (171.5,119.7) .. controls (171.74,119.46) and (171.8,118.85) .. (171.5,118.7) .. controls (170.66,118.28) and (168.83,120.03) .. (169.5,120.7) .. controls (170.83,122.03) and (172.83,120.03) .. (171.5,118.7) .. controls (168.28,115.48) and (170.47,122.67) .. (168.5,120.7) .. controls (166.72,118.92) and (172.28,120.48) .. (170.5,118.7) .. controls (169.53,117.73) and (168.89,121.47) .. (169.5,122.7) .. controls (169.8,123.3) and (171.03,123.17) .. (171.5,122.7) .. controls (171.97,122.23) and (172.17,120.7) .. (171.5,120.7) ;
\draw  [line width=3] [line join = round][line cap = round] (173.5,120.7) .. controls (173.17,120.7) and (172.83,120.7) .. (172.5,120.7) ;
\draw  [line width=3] [line join = round][line cap = round] (172.77,118.84) .. controls (172.77,118.84) and (172.77,118.84) .. (172.77,118.84) ;
\draw  [line width=3] [line join = round][line cap = round] (168.77,118.84) .. controls (168.77,118.84) and (168.77,118.84) .. (168.77,118.84) ;
\draw  [line width=3] [line join = round][line cap = round] (168.77,121.84) .. controls (168.77,121.84) and (168.77,121.84) .. (168.77,121.84) ;
\draw  [line width=3] [line join = round][line cap = round] (169.77,119.84) .. controls (169.77,119.84) and (169.77,119.84) .. (169.77,119.84) ;
\draw  [line width=3] [line join = round][line cap = round] (240.5,190.6) .. controls (240.5,190.6) and (240.5,190.6) .. (240.5,190.6) ;
\draw  [line width=3] [line join = round][line cap = round] (240.27,50.34) .. controls (240.27,50.34) and (240.27,50.34) .. (240.27,50.34) ;
\draw  [line width=3] [line join = round][line cap = round] (240.5,190) .. controls (240.5,190) and (240.5,190) .. (240.5,190) ;
\draw   (245,190) .. controls (245,187.51) and (242.99,185.5) .. (240.5,185.5) .. controls (238.01,185.5) and (236,187.51) .. (236,190) .. controls (236,192.49) and (238.01,194.5) .. (240.5,194.5) .. controls (242.99,194.5) and (245,192.49) .. (245,190) -- cycle ;
\draw  [line width=3] [line join = round][line cap = round] (239.5,189.2) .. controls (240.17,189.2) and (241.03,189.67) .. (241.5,189.2) .. controls (241.74,188.96) and (241.8,188.35) .. (241.5,188.2) .. controls (240.66,187.78) and (238.83,189.53) .. (239.5,190.2) .. controls (240.83,191.53) and (242.83,189.53) .. (241.5,188.2) .. controls (238.28,184.98) and (240.47,192.17) .. (238.5,190.2) .. controls (236.72,188.42) and (242.28,189.98) .. (240.5,188.2) .. controls (239.53,187.23) and (238.89,190.97) .. (239.5,192.2) .. controls (239.8,192.8) and (241.03,192.67) .. (241.5,192.2) .. controls (241.97,191.73) and (242.17,190.2) .. (241.5,190.2) ;
\draw  [line width=3] [line join = round][line cap = round] (243.5,190.2) .. controls (243.17,190.2) and (242.83,190.2) .. (242.5,190.2) ;
\draw  [line width=3] [line join = round][line cap = round] (242.77,188.34) .. controls (242.77,188.34) and (242.77,188.34) .. (242.77,188.34) ;
\draw  [line width=3] [line join = round][line cap = round] (238.77,188.34) .. controls (238.77,188.34) and (238.77,188.34) .. (238.77,188.34) ;
\draw  [line width=3] [line join = round][line cap = round] (238.77,191.34) .. controls (238.77,191.34) and (238.77,191.34) .. (238.77,191.34) ;
\draw  [line width=3] [line join = round][line cap = round] (239.77,189.34) .. controls (239.77,189.34) and (239.77,189.34) .. (239.77,189.34) ;
\draw  [line width=3] [line join = round][line cap = round] (240,50) .. controls (240,50) and (240,50) .. (240,50) ;
\draw   (244.5,50) .. controls (244.5,47.51) and (242.49,45.5) .. (240,45.5) .. controls (237.51,45.5) and (235.5,47.51) .. (235.5,50) .. controls (235.5,52.49) and (237.51,54.5) .. (240,54.5) .. controls (242.49,54.5) and (244.5,52.49) .. (244.5,50) -- cycle ;
\draw  [line width=3] [line join = round][line cap = round] (239,49.2) .. controls (239.67,49.2) and (240.53,49.67) .. (241,49.2) .. controls (241.24,48.96) and (241.3,48.35) .. (241,48.2) .. controls (240.16,47.78) and (238.33,49.53) .. (239,50.2) .. controls (240.33,51.53) and (242.33,49.53) .. (241,48.2) .. controls (237.78,44.98) and (239.97,52.17) .. (238,50.2) .. controls (236.22,48.42) and (241.78,49.98) .. (240,48.2) .. controls (239.03,47.23) and (238.39,50.97) .. (239,52.2) .. controls (239.3,52.8) and (240.53,52.67) .. (241,52.2) .. controls (241.47,51.73) and (241.67,50.2) .. (241,50.2) ;
\draw  [line width=3] [line join = round][line cap = round] (243,50.2) .. controls (242.67,50.2) and (242.33,50.2) .. (242,50.2) ;
\draw  [line width=3] [line join = round][line cap = round] (242.27,48.34) .. controls (242.27,48.34) and (242.27,48.34) .. (242.27,48.34) ;
\draw  [line width=3] [line join = round][line cap = round] (238.27,48.34) .. controls (238.27,48.34) and (238.27,48.34) .. (238.27,48.34) ;
\draw  [line width=3] [line join = round][line cap = round] (238.27,51.34) .. controls (238.27,51.34) and (238.27,51.34) .. (238.27,51.34) ;
\draw  [line width=3] [line join = round][line cap = round] (239.27,49.34) .. controls (239.27,49.34) and (239.27,49.34) .. (239.27,49.34) ;
\draw    (240,50) -- (310,120) ;
\draw    (240,50) -- (170,120) ;
\draw    (240,190) -- (170,120) ;
\draw    (310,120) -- (240,190) ;
\draw (312,90.4) node [anchor=north west][inner sep=0.75pt]    {$x_{4}$};
\draw (142,112.4) node [anchor=north west][inner sep=0.75pt]    {$x_{2}$};
\draw (352,129.9) node [anchor=north west][inner sep=0.75pt]    {$y_{1}$};
\draw (402.5,130.4) node [anchor=north west][inner sep=0.75pt]    {$y_{2}$};
\draw (543.76,129.58) node [anchor=north west][inner sep=0.75pt]    {$y_{n}{}_{-1}$};
\draw (602.5,130.4) node [anchor=north west][inner sep=0.75pt]    {$y_{n}$};
\draw (211,31.4) node [anchor=north west][inner sep=0.75pt]    {$x_{1}$};
\draw (211.5,186.4) node [anchor=north west][inner sep=0.75pt]    {$x_{3}$};
\end{tikzpicture}
    \caption{The tadpole $T_{4,n}$}
    \label{fig4}
\end{figure}
\begin{proof}
    From Proposition \ref{comp_inde_graph}, see Figure \ref{fig4}, we have
    \begin{align*}
        I\br{T_{4,n};t} &= I\br{T_{4,n}\setminus x_1;t}+ tI\br{T_{4,n}\setminus N\brc{x_1};t}\\
        &= I\br{P_{n+3};t}+ t(1+t)I\br{P_n;t}\\
        &= \br{I\br{P_{n+2};t}+tI\br{P_{n+1};t}}+ tI\br{P_n;t}+t^2I\br{P_n;t}\\
        &= \br{I\br{P_{n+2};t}+ tI\br{P_n;t}}+ t\br{I\br{P_{n+1};t}+ tI\br{P_n;t}}\\
        &= I\br{C_{n+3};t}+ tI\br{P_{n+2};t}.
    \end{align*}

Applying Proposition \ref{compare_modes}(ii), we have two following cases:\\
\underline{\textbf{Case 1.} $\rho_{n+3}=\lambda_{n+2}+1$.} Both of polynomials $I\br{C_{n+3};t}$ and $tI\br{P_{n+2};t}$ are unimodal with mode $\lambda_{n+2}+1$. Therefore $I\br{T_{4,n};t}$ is unimodal with the mode $\lambda_{n+2}+1$.\\
\underline{\textbf{Case 2.} $\rho_{n+3}=\lambda_{n+2}$.} $I\br{C_{n+3};t}$ has the mode $\rho_{n+3}=\lambda_{n+2}$, while $tI\br{P_{n+2};t}$ has the mode $\lambda_{n+2}+1$. Therefore, applying Lemma \ref{unimodal_mode_unit}, $I\br{T_{4,n};t}$ is unimodal whose mode beglongs to $\brb{\lambda_{n+2},\lambda_{n+2}+1}$.

In conclusion, we have $I\br{T_{4,n};t}$ is unimodal whose mode belongs to $\brb{\lambda_{n+2},\lambda_{n+2}+1}$.
\end{proof}

\begin{prop}
    $I\br{T_{5,n};t}$ is unimodal with the mode belongs to $\brb{\lambda_{n+2},\lambda_{n+2}+1,\lambda_{n+2}+2}$, for all $n\ge 5$.
\end{prop}
\begin{figure}[!ht]
    \centering
\tikzset{every picture/.style={line width=0.75pt}} 
\begin{tikzpicture}[x=0.75pt,y=0.75pt,yscale=-1,xscale=1]
\draw  [line width=3] [line join = round][line cap = round] (310,120.6) .. controls (310,120.6) and (310,120.6) .. (310,120.6) ;
\draw    (360,120) -- (391.76,120.18) -- (410,120) ;
\draw  [line width=3] [line join = round][line cap = round] (410,120) .. controls (410,120) and (410,120) .. (410,120) ;
\draw  [line width=3] [line join = round][line cap = round] (204.77,86.84) .. controls (204.77,86.84) and (204.77,86.84) .. (204.77,86.84) ;
\draw    (310,120) -- (360,120) ;
\draw  [line width=3] [line join = round][line cap = round] (360,120) .. controls (360,120) and (360,120) .. (360,120) ;
\draw  [line width=3] [line join = round][line cap = round] (268.27,174.34) .. controls (268.27,174.34) and (268.27,174.34) .. (268.27,174.34) ;
\draw    (460,120) -- (491.76,120.18) -- (510,120) ;
\draw  [line width=3] [line join = round][line cap = round] (610,120) .. controls (610,120) and (610,120) .. (610,120) ;
\draw    (410,120) -- (441.76,120.18) -- (460,120) ;
\draw  [line width=3] [line join = round][line cap = round] (560,120) .. controls (560,120) and (560,120) .. (560,120) ;
\draw  [line width=3] [line join = round][line cap = round] (460,120) .. controls (460,120) and (460,120) .. (460,120) ;
\draw  [line width=3] [line join = round][line cap = round] (511,120) .. controls (511,120) and (511,120) .. (511,120) ;
\draw    (510,120) -- (541.76,120.18) -- (560,120) ;
\draw    (560,120) -- (591.76,120.18) -- (610,120) ;
\draw   (364.5,120) .. controls (364.5,117.51) and (362.49,115.5) .. (360,115.5) .. controls (357.51,115.5) and (355.5,117.51) .. (355.5,120) .. controls (355.5,122.49) and (357.51,124.5) .. (360,124.5) .. controls (362.49,124.5) and (364.5,122.49) .. (364.5,120) -- cycle ;
\draw  [line width=3] [line join = round][line cap = round] (359,119.2) .. controls (359.67,119.2) and (360.53,119.67) .. (361,119.2) .. controls (361.24,118.96) and (361.3,118.35) .. (361,118.2) .. controls (360.16,117.78) and (358.33,119.53) .. (359,120.2) .. controls (360.33,121.53) and (362.33,119.53) .. (361,118.2) .. controls (357.78,114.98) and (359.97,122.17) .. (358,120.2) .. controls (356.22,118.42) and (361.78,119.98) .. (360,118.2) .. controls (359.03,117.23) and (358.39,120.97) .. (359,122.2) .. controls (359.3,122.8) and (360.53,122.67) .. (361,122.2) .. controls (361.47,121.73) and (361.67,120.2) .. (361,120.2) ;
\draw  [line width=3] [line join = round][line cap = round] (363,120.2) .. controls (362.67,120.2) and (362.33,120.2) .. (362,120.2) ;
\draw  [line width=3] [line join = round][line cap = round] (362.28,118.34) .. controls (362.28,118.34) and (362.28,118.34) .. (362.28,118.34) ;
\draw  [line width=3] [line join = round][line cap = round] (358.28,118.34) .. controls (358.28,118.34) and (358.28,118.34) .. (358.28,118.34) ;
\draw  [line width=3] [line join = round][line cap = round] (358.28,121.34) .. controls (358.28,121.34) and (358.28,121.34) .. (358.28,121.34) ;
\draw  [line width=3] [line join = round][line cap = round] (359.28,119.34) .. controls (359.28,119.34) and (359.28,119.34) .. (359.28,119.34) ;
\draw   (414.5,120) .. controls (414.5,117.51) and (412.49,115.5) .. (410,115.5) .. controls (407.51,115.5) and (405.5,117.51) .. (405.5,120) .. controls (405.5,122.49) and (407.51,124.5) .. (410,124.5) .. controls (412.49,124.5) and (414.5,122.49) .. (414.5,120) -- cycle ;
\draw  [line width=3] [line join = round][line cap = round] (409,119.2) .. controls (409.67,119.2) and (410.53,119.67) .. (411,119.2) .. controls (411.24,118.96) and (411.3,118.35) .. (411,118.2) .. controls (410.16,117.78) and (408.33,119.53) .. (409,120.2) .. controls (410.33,121.53) and (412.33,119.53) .. (411,118.2) .. controls (407.78,114.98) and (409.97,122.17) .. (408,120.2) .. controls (406.22,118.42) and (411.78,119.98) .. (410,118.2) .. controls (409.03,117.23) and (408.39,120.97) .. (409,122.2) .. controls (409.3,122.8) and (410.53,122.67) .. (411,122.2) .. controls (411.47,121.73) and (411.67,120.2) .. (411,120.2) ;
\draw  [line width=3] [line join = round][line cap = round] (413,120.2) .. controls (412.67,120.2) and (412.33,120.2) .. (412,120.2) ;
\draw  [line width=3] [line join = round][line cap = round] (412.28,118.34) .. controls (412.28,118.34) and (412.28,118.34) .. (412.28,118.34) ;
\draw  [line width=3] [line join = round][line cap = round] (408.28,118.34) .. controls (408.28,118.34) and (408.28,118.34) .. (408.28,118.34) ;
\draw  [line width=3] [line join = round][line cap = round] (408.28,121.34) .. controls (408.28,121.34) and (408.28,121.34) .. (408.28,121.34) ;
\draw  [line width=3] [line join = round][line cap = round] (409.28,119.34) .. controls (409.28,119.34) and (409.28,119.34) .. (409.28,119.34) ;
\draw   (464.5,120) .. controls (464.5,117.51) and (462.49,115.5) .. (460,115.5) .. controls (457.51,115.5) and (455.5,117.51) .. (455.5,120) .. controls (455.5,122.49) and (457.51,124.5) .. (460,124.5) .. controls (462.49,124.5) and (464.5,122.49) .. (464.5,120) -- cycle ;
\draw  [line width=3] [line join = round][line cap = round] (459,119.2) .. controls (459.67,119.2) and (460.53,119.67) .. (461,119.2) .. controls (461.24,118.96) and (461.3,118.35) .. (461,118.2) .. controls (460.16,117.78) and (458.33,119.53) .. (459,120.2) .. controls (460.33,121.53) and (462.33,119.53) .. (461,118.2) .. controls (457.78,114.98) and (459.97,122.17) .. (458,120.2) .. controls (456.22,118.42) and (461.78,119.98) .. (460,118.2) .. controls (459.03,117.23) and (458.39,120.97) .. (459,122.2) .. controls (459.3,122.8) and (460.53,122.67) .. (461,122.2) .. controls (461.47,121.73) and (461.67,120.2) .. (461,120.2) ;
\draw  [line width=3] [line join = round][line cap = round] (463,120.2) .. controls (462.67,120.2) and (462.33,120.2) .. (462,120.2) ;
\draw  [line width=3] [line join = round][line cap = round] (462.28,118.34) .. controls (462.28,118.34) and (462.28,118.34) .. (462.28,118.34) ;
\draw  [line width=3] [line join = round][line cap = round] (458.28,118.34) .. controls (458.28,118.34) and (458.28,118.34) .. (458.28,118.34) ;
\draw  [line width=3] [line join = round][line cap = round] (458.28,121.34) .. controls (458.28,121.34) and (458.28,121.34) .. (458.28,121.34) ;
\draw  [line width=3] [line join = round][line cap = round] (459.28,119.34) .. controls (459.28,119.34) and (459.28,119.34) .. (459.28,119.34) ;
\draw  [line width=3] [line join = round][line cap = round] (610,120) .. controls (610,120) and (610,120) .. (610,120) ;
\draw   (614.5,120) .. controls (614.5,117.51) and (612.49,115.5) .. (610,115.5) .. controls (607.51,115.5) and (605.5,117.51) .. (605.5,120) .. controls (605.5,122.49) and (607.51,124.5) .. (610,124.5) .. controls (612.49,124.5) and (614.5,122.49) .. (614.5,120) -- cycle ;
\draw  [line width=3] [line join = round][line cap = round] (609,119.2) .. controls (609.67,119.2) and (610.53,119.67) .. (611,119.2) .. controls (611.24,118.96) and (611.3,118.35) .. (611,118.2) .. controls (610.16,117.78) and (608.33,119.53) .. (609,120.2) .. controls (610.33,121.53) and (612.33,119.53) .. (611,118.2) .. controls (607.78,114.98) and (609.97,122.17) .. (608,120.2) .. controls (606.22,118.42) and (611.78,119.98) .. (610,118.2) .. controls (609.03,117.23) and (608.39,120.97) .. (609,122.2) .. controls (609.3,122.8) and (610.53,122.67) .. (611,122.2) .. controls (611.47,121.73) and (611.67,120.2) .. (611,120.2) ;
\draw  [line width=3] [line join = round][line cap = round] (613,120.2) .. controls (612.67,120.2) and (612.33,120.2) .. (612,120.2) ;
\draw  [line width=3] [line join = round][line cap = round] (612.28,118.34) .. controls (612.28,118.34) and (612.28,118.34) .. (612.28,118.34) ;
\draw  [line width=3] [line join = round][line cap = round] (608.28,118.34) .. controls (608.28,118.34) and (608.28,118.34) .. (608.28,118.34) ;
\draw  [line width=3] [line join = round][line cap = round] (608.28,121.34) .. controls (608.28,121.34) and (608.28,121.34) .. (608.28,121.34) ;
\draw  [line width=3] [line join = round][line cap = round] (609.28,119.34) .. controls (609.28,119.34) and (609.28,119.34) .. (609.28,119.34) ;
\draw  [line width=3] [line join = round][line cap = round] (510,120) .. controls (510,120) and (510,120) .. (510,120) ;
\draw   (514.5,120) .. controls (514.5,117.51) and (512.49,115.5) .. (510,115.5) .. controls (507.51,115.5) and (505.5,117.51) .. (505.5,120) .. controls (505.5,122.49) and (507.51,124.5) .. (510,124.5) .. controls (512.49,124.5) and (514.5,122.49) .. (514.5,120) -- cycle ;
\draw  [line width=3] [line join = round][line cap = round] (509,119.2) .. controls (509.67,119.2) and (510.53,119.67) .. (511,119.2) .. controls (511.24,118.96) and (511.3,118.35) .. (511,118.2) .. controls (510.16,117.78) and (508.33,119.53) .. (509,120.2) .. controls (510.33,121.53) and (512.33,119.53) .. (511,118.2) .. controls (507.78,114.98) and (509.97,122.17) .. (508,120.2) .. controls (506.22,118.42) and (511.78,119.98) .. (510,118.2) .. controls (509.03,117.23) and (508.39,120.97) .. (509,122.2) .. controls (509.3,122.8) and (510.53,122.67) .. (511,122.2) .. controls (511.47,121.73) and (511.67,120.2) .. (511,120.2) ;
\draw  [line width=3] [line join = round][line cap = round] (513,120.2) .. controls (512.67,120.2) and (512.33,120.2) .. (512,120.2) ;
\draw  [line width=3] [line join = round][line cap = round] (512.28,118.34) .. controls (512.28,118.34) and (512.28,118.34) .. (512.28,118.34) ;
\draw  [line width=3] [line join = round][line cap = round] (508.28,118.34) .. controls (508.28,118.34) and (508.28,118.34) .. (508.28,118.34) ;
\draw  [line width=3] [line join = round][line cap = round] (508.28,121.34) .. controls (508.28,121.34) and (508.28,121.34) .. (508.28,121.34) ;
\draw  [line width=3] [line join = round][line cap = round] (509.28,119.34) .. controls (509.28,119.34) and (509.28,119.34) .. (509.28,119.34) ;
\draw  [line width=3] [line join = round][line cap = round] (560,120) .. controls (560,120) and (560,120) .. (560,120) ;
\draw   (564.5,120) .. controls (564.5,117.51) and (562.49,115.5) .. (560,115.5) .. controls (557.51,115.5) and (555.5,117.51) .. (555.5,120) .. controls (555.5,122.49) and (557.51,124.5) .. (560,124.5) .. controls (562.49,124.5) and (564.5,122.49) .. (564.5,120) -- cycle ;
\draw  [line width=3] [line join = round][line cap = round] (559,119.2) .. controls (559.67,119.2) and (560.53,119.67) .. (561,119.2) .. controls (561.24,118.96) and (561.3,118.35) .. (561,118.2) .. controls (560.16,117.78) and (558.33,119.53) .. (559,120.2) .. controls (560.33,121.53) and (562.33,119.53) .. (561,118.2) .. controls (557.78,114.98) and (559.97,122.17) .. (558,120.2) .. controls (556.22,118.42) and (561.78,119.98) .. (560,118.2) .. controls (559.03,117.23) and (558.39,120.97) .. (559,122.2) .. controls (559.3,122.8) and (560.53,122.67) .. (561,122.2) .. controls (561.47,121.73) and (561.67,120.2) .. (561,120.2) ;
\draw  [line width=3] [line join = round][line cap = round] (563,120.2) .. controls (562.67,120.2) and (562.33,120.2) .. (562,120.2) ;
\draw  [line width=3] [line join = round][line cap = round] (562.28,118.34) .. controls (562.28,118.34) and (562.28,118.34) .. (562.28,118.34) ;
\draw  [line width=3] [line join = round][line cap = round] (558.28,118.34) .. controls (558.28,118.34) and (558.28,118.34) .. (558.28,118.34) ;
\draw  [line width=3] [line join = round][line cap = round] (558.28,121.34) .. controls (558.28,121.34) and (558.28,121.34) .. (558.28,121.34) ;
\draw  [line width=3] [line join = round][line cap = round] (559.28,119.34) .. controls (559.28,119.34) and (559.28,119.34) .. (559.28,119.34) ;
\draw  [line width=3] [line join = round][line cap = round] (310,120) .. controls (310,120) and (310,120) .. (310,120) ;
\draw   (314.5,120) .. controls (314.5,117.51) and (312.49,115.5) .. (310,115.5) .. controls (307.51,115.5) and (305.5,117.51) .. (305.5,120) .. controls (305.5,122.49) and (307.51,124.5) .. (310,124.5) .. controls (312.49,124.5) and (314.5,122.49) .. (314.5,120) -- cycle ;
\draw  [line width=3] [line join = round][line cap = round] (309,119.2) .. controls (309.67,119.2) and (310.53,119.67) .. (311,119.2) .. controls (311.24,118.96) and (311.3,118.35) .. (311,118.2) .. controls (310.16,117.78) and (308.33,119.53) .. (309,120.2) .. controls (310.33,121.53) and (312.33,119.53) .. (311,118.2) .. controls (307.78,114.98) and (309.97,122.17) .. (308,120.2) .. controls (306.22,118.42) and (311.78,119.98) .. (310,118.2) .. controls (309.03,117.23) and (308.39,120.97) .. (309,122.2) .. controls (309.3,122.8) and (310.53,122.67) .. (311,122.2) .. controls (311.47,121.73) and (311.67,120.2) .. (311,120.2) ;
\draw  [line width=3] [line join = round][line cap = round] (313,120.2) .. controls (312.67,120.2) and (312.33,120.2) .. (312,120.2) ;
\draw  [line width=3] [line join = round][line cap = round] (312.27,118.34) .. controls (312.27,118.34) and (312.27,118.34) .. (312.27,118.34) ;
\draw  [line width=3] [line join = round][line cap = round] (308.27,118.34) .. controls (308.27,118.34) and (308.27,118.34) .. (308.27,118.34) ;
\draw  [line width=3] [line join = round][line cap = round] (308.27,121.34) .. controls (308.27,121.34) and (308.27,121.34) .. (308.27,121.34) ;
\draw  [line width=3] [line join = round][line cap = round] (309.27,119.34) .. controls (309.27,119.34) and (309.27,119.34) .. (309.27,119.34) ;
\draw  [line width=3] [line join = round][line cap = round] (268,174) .. controls (268,174) and (268,174) .. (268,174) ;
\draw   (272.5,174) .. controls (272.5,171.51) and (270.49,169.5) .. (268,169.5) .. controls (265.51,169.5) and (263.5,171.51) .. (263.5,174) .. controls (263.5,176.49) and (265.51,178.5) .. (268,178.5) .. controls (270.49,178.5) and (272.5,176.49) .. (272.5,174) -- cycle ;
\draw  [line width=3] [line join = round][line cap = round] (267,173.2) .. controls (267.67,173.2) and (268.53,173.67) .. (269,173.2) .. controls (269.24,172.96) and (269.3,172.35) .. (269,172.2) .. controls (268.16,171.78) and (266.33,173.53) .. (267,174.2) .. controls (268.33,175.53) and (270.33,173.53) .. (269,172.2) .. controls (265.78,168.98) and (267.97,176.17) .. (266,174.2) .. controls (264.22,172.42) and (269.78,173.98) .. (268,172.2) .. controls (267.03,171.23) and (266.39,174.97) .. (267,176.2) .. controls (267.3,176.8) and (268.53,176.67) .. (269,176.2) .. controls (269.47,175.73) and (269.67,174.2) .. (269,174.2) ;
\draw  [line width=3] [line join = round][line cap = round] (271,174.2) .. controls (270.67,174.2) and (270.33,174.2) .. (270,174.2) ;
\draw  [line width=3] [line join = round][line cap = round] (270.27,172.34) .. controls (270.27,172.34) and (270.27,172.34) .. (270.27,172.34) ;
\draw  [line width=3] [line join = round][line cap = round] (266.27,172.34) .. controls (266.27,172.34) and (266.27,172.34) .. (266.27,172.34) ;
\draw  [line width=3] [line join = round][line cap = round] (266.27,175.34) .. controls (266.27,175.34) and (266.27,175.34) .. (266.27,175.34) ;
\draw  [line width=3] [line join = round][line cap = round] (267.27,173.34) .. controls (267.27,173.34) and (267.27,173.34) .. (267.27,173.34) ;
\draw  [line width=3] [line join = round][line cap = round] (204.5,86.5) .. controls (204.5,86.5) and (204.5,86.5) .. (204.5,86.5) ;
\draw   (209,86.5) .. controls (209,84.01) and (206.99,82) .. (204.5,82) .. controls (202.01,82) and (200,84.01) .. (200,86.5) .. controls (200,88.99) and (202.01,91) .. (204.5,91) .. controls (206.99,91) and (209,88.99) .. (209,86.5) -- cycle ;
\draw  [line width=3] [line join = round][line cap = round] (203.5,85.7) .. controls (204.17,85.7) and (205.03,86.17) .. (205.5,85.7) .. controls (205.74,85.46) and (205.8,84.85) .. (205.5,84.7) .. controls (204.66,84.28) and (202.83,86.03) .. (203.5,86.7) .. controls (204.83,88.03) and (206.83,86.03) .. (205.5,84.7) .. controls (202.28,81.48) and (204.47,88.67) .. (202.5,86.7) .. controls (200.72,84.92) and (206.28,86.48) .. (204.5,84.7) .. controls (203.53,83.73) and (202.89,87.47) .. (203.5,88.7) .. controls (203.8,89.3) and (205.03,89.17) .. (205.5,88.7) .. controls (205.97,88.23) and (206.17,86.7) .. (205.5,86.7) ;
\draw  [line width=3] [line join = round][line cap = round] (207.5,86.7) .. controls (207.17,86.7) and (206.83,86.7) .. (206.5,86.7) ;
\draw  [line width=3] [line join = round][line cap = round] (206.77,84.84) .. controls (206.77,84.84) and (206.77,84.84) .. (206.77,84.84) ;
\draw  [line width=3] [line join = round][line cap = round] (202.77,84.84) .. controls (202.77,84.84) and (202.77,84.84) .. (202.77,84.84) ;
\draw  [line width=3] [line join = round][line cap = round] (202.77,87.84) .. controls (202.77,87.84) and (202.77,87.84) .. (202.77,87.84) ;
\draw  [line width=3] [line join = round][line cap = round] (203.77,85.84) .. controls (203.77,85.84) and (203.77,85.84) .. (203.77,85.84) ;
\draw  [line width=3] [line join = round][line cap = round] (204,154.6) .. controls (204,154.6) and (204,154.6) .. (204,154.6) ;
\draw  [line width=3] [line join = round][line cap = round] (268.27,65.34) .. controls (268.27,65.34) and (268.27,65.34) .. (268.27,65.34) ;
\draw  [line width=3] [line join = round][line cap = round] (204,154) .. controls (204,154) and (204,154) .. (204,154) ;
\draw   (208.5,154) .. controls (208.5,151.51) and (206.49,149.5) .. (204,149.5) .. controls (201.51,149.5) and (199.5,151.51) .. (199.5,154) .. controls (199.5,156.49) and (201.51,158.5) .. (204,158.5) .. controls (206.49,158.5) and (208.5,156.49) .. (208.5,154) -- cycle ;
\draw  [line width=3] [line join = round][line cap = round] (203,153.2) .. controls (203.67,153.2) and (204.53,153.67) .. (205,153.2) .. controls (205.24,152.96) and (205.3,152.35) .. (205,152.2) .. controls (204.16,151.78) and (202.33,153.53) .. (203,154.2) .. controls (204.33,155.53) and (206.33,153.53) .. (205,152.2) .. controls (201.78,148.98) and (203.97,156.17) .. (202,154.2) .. controls (200.22,152.42) and (205.78,153.98) .. (204,152.2) .. controls (203.03,151.23) and (202.39,154.97) .. (203,156.2) .. controls (203.3,156.8) and (204.53,156.67) .. (205,156.2) .. controls (205.47,155.73) and (205.67,154.2) .. (205,154.2) ;
\draw  [line width=3] [line join = round][line cap = round] (207,154.2) .. controls (206.67,154.2) and (206.33,154.2) .. (206,154.2) ;
\draw  [line width=3] [line join = round][line cap = round] (206.27,152.34) .. controls (206.27,152.34) and (206.27,152.34) .. (206.27,152.34) ;
\draw  [line width=3] [line join = round][line cap = round] (202.27,152.34) .. controls (202.27,152.34) and (202.27,152.34) .. (202.27,152.34) ;
\draw  [line width=3] [line join = round][line cap = round] (202.27,155.34) .. controls (202.27,155.34) and (202.27,155.34) .. (202.27,155.34) ;
\draw  [line width=3] [line join = round][line cap = round] (203.27,153.34) .. controls (203.27,153.34) and (203.27,153.34) .. (203.27,153.34) ;
\draw  [line width=3] [line join = round][line cap = round] (268,65) .. controls (268,65) and (268,65) .. (268,65) ;
\draw   (272.5,65) .. controls (272.5,62.51) and (270.49,60.5) .. (268,60.5) .. controls (265.51,60.5) and (263.5,62.51) .. (263.5,65) .. controls (263.5,67.49) and (265.51,69.5) .. (268,69.5) .. controls (270.49,69.5) and (272.5,67.49) .. (272.5,65) -- cycle ;
\draw  [line width=3] [line join = round][line cap = round] (267,64.2) .. controls (267.67,64.2) and (268.53,64.67) .. (269,64.2) .. controls (269.24,63.96) and (269.3,63.35) .. (269,63.2) .. controls (268.16,62.78) and (266.33,64.53) .. (267,65.2) .. controls (268.33,66.53) and (270.33,64.53) .. (269,63.2) .. controls (265.78,59.98) and (267.97,67.17) .. (266,65.2) .. controls (264.22,63.42) and (269.78,64.98) .. (268,63.2) .. controls (267.03,62.23) and (266.39,65.97) .. (267,67.2) .. controls (267.3,67.8) and (268.53,67.67) .. (269,67.2) .. controls (269.47,66.73) and (269.67,65.2) .. (269,65.2) ;
\draw  [line width=3] [line join = round][line cap = round] (271,65.2) .. controls (270.67,65.2) and (270.33,65.2) .. (270,65.2) ;
\draw  [line width=3] [line join = round][line cap = round] (270.27,63.34) .. controls (270.27,63.34) and (270.27,63.34) .. (270.27,63.34) ;
\draw  [line width=3] [line join = round][line cap = round] (266.27,63.34) .. controls (266.27,63.34) and (266.27,63.34) .. (266.27,63.34) ; 
\draw  [line width=3] [line join = round][line cap = round] (266.27,66.34) .. controls (266.27,66.34) and (266.27,66.34) .. (266.27,66.34) ;
\draw  [line width=3] [line join = round][line cap = round] (267.27,64.34) .. controls (267.27,64.34) and (267.27,64.34) .. (267.27,64.34) ;
\draw   (310,120) -- (269.44,175.83) -- (203.81,154.5) -- (203.81,85.5) -- (269.44,64.17) -- cycle ;
\draw (312,90.4) node [anchor=north west][inner sep=0.75pt]    {$x_{5}$};
\draw (172,74.4) node [anchor=north west][inner sep=0.75pt]    {$x_{2}$};
\draw (282,169.4) node [anchor=north west][inner sep=0.75pt]    {$x_{4}$};
\draw (352,129.9) node [anchor=north west][inner sep=0.75pt]    {$y_{1}$};
\draw (402.5,130.4) node [anchor=north west][inner sep=0.75pt]    {$y_{2}$};
\draw (543.76,129.58) node [anchor=north west][inner sep=0.75pt]    {$y_{n}{}_{-1}$};
\draw (602.5,130.4) node [anchor=north west][inner sep=0.75pt]    {$y_{n}$};
\draw (282,48.4) node [anchor=north west][inner sep=0.75pt]    {$x_{1}$};
\draw (172,150.4) node [anchor=north west][inner sep=0.75pt]    {$x_{3}$};
\end{tikzpicture}
    \caption{The tadpole $T_{5,n}$}
    \label{fig5}
\end{figure}
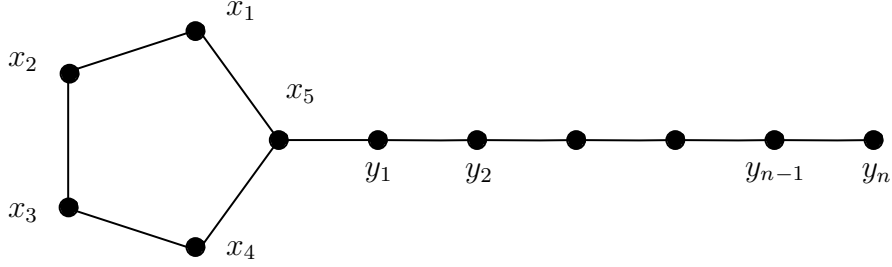
\begin{proof}
    From Proposition \ref{comp_inde_graph}, see Figure \ref{fig5}, we have
    \begin{align*}
        I\br{T_{5,n};t} &= I\br{T_{5,n}\setminus x_1;t}+ tI\br{T_{5,n}\setminus N\brc{x_1};t}\\
        &= I\br{P_{n+4};t}+ t(1+2t)I\br{P_n;t}\\
        &= \br{I\br{P_{n+3};t}+ tI\br{P_{n+2};t}}+ t(1+2t)I\br{P_n;t}\\
        &= \br{I\br{P_{n+3};t}+ t(1+t)I\br{P_n;t}} + t\br{I\br{P_{n+2};t}+ tI\br{P_n;t}}\\
        &= I\br{T_{4,n};t}+ tI\br{C_{n+3};t}.
    \end{align*}

Applying Proposition \ref{compare_modes}(ii), we have two following cases:\\
\underline{\textbf{Case 1.} $\rho_{n+3}=\lambda_{n+2}+1$.} From Proposition \ref{T_4,n}, $I\br{T_{4,n};t}$ is unimodal with the mode $\lambda_{n+2}+1$. Hence, applying Lemma \ref{unimodal_mode_unit}, $I\br{T_{5,n};t}$ is unimodal whose mode belongs to $\brb{\lambda_{n+2}+1,\lambda_{n+2}+2}$.\\
\underline{\textbf{Case 2.} $\rho_{n+3}=\lambda_{n+2}$.} From Proposition \ref{T_4,n}, $I\br{T_{4,n};t}$ is unimodal whose mode belongs to $\brb{\lambda_{n+2},\lambda_{n+2}+1}$. Applying Lemma \ref{unimodal_mode_unit}, $I\br{T_{5,n};t}$ is unimodal whose mode belongs to $\brb{\lambda_{n+2},\lambda_{n+2}+1}$.

In conclusion, $I\br{T_{5,n};t}$ is unimodal with the mode belongs to $\brb{\lambda_{n+2},\lambda_{n+2}+1,\lambda_{n+2}+2}$.
\end{proof}

\section{WLP for algebras associated to tadpole graphs}
In this section, we study the WLP for artinian monomial algebras associated to certain tadpole graphs. 
For a tadpole graph $T_{m,n}$, with $m\geq 3,n\geq 1$, we consider 
\begin{align*}
R= \Bbbk\brc{x_1,x_2,\ldots,x_m,y_1,y_2,\ldots,y_n},
\end{align*} 
$I(T_{m,n})\subset R$ is the edge ideal of $T_{m,n}$, $I=\br{x_1^2,\ldots,x_m^2,y_1^2,\ldots,y_n^2}+I\br{T_{m,n}}$. Then, $A\br{T_{m,n}}$ $=R/I$ is artinian monomial algebra associated to $T_{m,n}$. From now on, we always assume that the field $\Bbbk$ is of characteristic zero and denote by $\ell$ the sum of variables in the polynomial ring we are working with.
\begin{thm}\label{wlp_Tm2}
$A\br{T_{m,2}}$ has the WLP if and only if $m\in \brb{4,5,7,8,11}$.
\end{thm}
\begin{proof}
By using \texttt{Macaulay2} \cite{Macaulay2}, we can check that for $3\leq m\leq 15$, $A\br{T_{m,2}}$ has the WLP if and only if $m\in \brb{4,5,7,8,11}$. Consider $m\geq 16$ and see Figure \ref{fig2}. By \Cref{modeT_{m,n}}, we consider the following two cases:\\
\underline{\textbf{Case 1.} $I\br{T_{m,2};t}$ has the mode $\rho_m$.} We have the exact sequence 
    \begin{center}
        ${\xymatrix{
        R/I \ar@{->>}[r] & R/\br{I+\br{x_{m-1}}} \ar[r] &0
        }}$
    \end{center}
    and $R/\br{I+\br{x_{m-1}}}\cong A\br{P_{m+1}}$. Because $m+1\geq 17$, applying \Cref{thm_Pn}, $R/\br{I+\br{x_{m-1}}}$ fails the surjectivity at $\lambda_{m+1}$. Note that $\lambda_{m+1}\geq \rho_m$. Hence, $R/I$ fails the surjectivity at $\lambda_{m+1}$. In other words, $A\br{T_{m,2}}$ fails the WLP in this case.\\
    \underline{\textbf{Case 2.} $I\br{T_{m,2};t}$ has the mode $\rho_m+1$.} If $\lambda_{m+1}\geq \rho_m+1$, we can prove similarly as in Case 1. Consider $\lambda_{m+1}< \rho_m+1$. Applying \Cref{thm_Pn}, we have $\lambda_{m+1}=\lambda_m=\rho_m$.\\
    \underline{\textbf{Subcase 2.1.} $\lambda_m=\lambda_{m-1}+1$.} Applying Proposition \ref{compare_modes}, we have $\lambda_{m-1}=\lambda_{m-2}=\lambda_{m-3}=\lambda_{m-4}$ and $\lambda_{m-4}=\lambda_{m-5}+1$. We have the exact sequence 
    \begin{center}
        ${\xymatrix{
        0 \ar[r] & R/\br{I:\br{x_2y_1}}(-2) \ar@{^{(}->}[r]^-{\cdot x_2y_1} & R/I
        }}$
    \end{center}
    and $R/\br{I:\br{x_2y_1}}\cong A\br{P_{m-4}}$. Because $m-4\geq12$, applying Theorem \ref{thm_Pn}, $\times \ell :\brc{A\br{P_{m-4}}}_{\lambda_{m-4}-1}\rightarrow \brc{A\br{P_{m-4}}}_{\lambda_{m-4}}$ is not injective. Note that $\lambda_{m-4}-1=\lambda_m-2=\rho_m-2$. Therefore, the map $\times \ell :\brc{R/I}_{\rho_m}\rightarrow \brc{R/I}_{\rho_m+1}$ is not injective. Hence, $A\br{T_{m,2}}$ fails the WLP in this case.\\
    \underline{\textbf{Subcase 2.2.} $\lambda_m=\lambda_{m-1}$.} If $m\geq18$ or $m=16$, consider the exact sequence
    \begin{center}
        ${\xymatrix{
        R/I \ar@{->>}[r] & R/\br{I+\br{x_{m}}} \ar[r] &0
        }}$
    \end{center}
     and $R/\br{I+\br{x_m}}\cong A\br{P_{m-1}}\otimes_\Bbbk A\br{P_2}$. We have $m-1\geq 17$ or $m-1=15$ so $A\br{P_{m-1}}$ fails the surjectivity at $\lambda_{m-1}$. On the other hand, $A\br{P_2}$ fails the surjectivity at 0. Then $A\br{P_{m-1}}\otimes_\Bbbk A\br{P_2}$ fails the surjectivity at $\lambda_{m-1}+1$ (Lemma \ref{tensor}). Therefore, the map $\times \ell :\brc{R/I}_{\rho_m+1}\rightarrow \brc{R/I}_{\rho_m+2}$ is not surjective. If $m=17$, the mode of $I\br{T_{17,2};t}$ is $\rho_{17}$ (return to Case 1).
\end{proof}

\begin{thm}\label{wlp_Tm3}
    $A\br{T_{m,3}}$ has the WLP if and only if $m\in \brb{3,4,5,6,7,8,10,11,14}$.
\end{thm}
\begin{proof}
    By using \texttt{Macaulay2}, we can check that for $3\leq m\leq 14$, $A\br{T_{m,3}}$ has the WLP if and only if $m\in \brb{3,4,5,6,7,8,10,11,14}$. Consider $m\geq 15$ and see Figure \ref{fig3}. \\
    \underline{\textbf{Case 1.} $I\br{T_{m,3};t}$ has the mode $\rho_m$.} Consider the exact sequence
    \begin{center}
        ${\xymatrix{
        R/I \ar@{->>}[r] & R/\br{I+\br{x_{m-1}}} \ar[r] &0
        }}$
    \end{center}
    and $R/\br{I+\br{x_{m-1}}}\cong A\br{P_{m+2}}$. Because $m+2\geq 17$, applying Theorem \ref{thm_Pn}, $R/\br{I+\br{x_{m-1}}}$ fails the surjectivity at $\lambda_{m+2}$. Note that $\lambda_{m+2}\geq \rho_m$. Hence, $R/I$ fails the surjectivity at $\lambda_{m+2}$. In other words, $A\br{T_{m,3}}$ fails the WLP in this case.\\
    \underline{\textbf{Case 2.} $I\br{T_{m,3};t}$ has the mode $\rho_m+1$.}\\
    \underline{\textbf{Subcase 2.1.} $\lambda_{m+2}\geq\rho_m+1$.} We prove similarly as in Case 1.\\
    \underline{\textbf{Subcase 2.2.} $\lambda_{m+2}\leq\rho_m\leq\lambda_{m-4}+1\leq\lambda_m$.} Then $\lambda_m=\lambda_{m+1}=\lambda_{m+2}=\rho_m.$
    
 If $\lambda_m=\lambda_{m-1}+1$, then we consider the exact sequence
        \begin{center}
        ${\xymatrix{
        0 \ar[r] & R/\br{I:\br{x_{m-2}}}(-1) \ar@{^{(}->}[r]^-{\cdot x_{m-2}} &R/I
        }}$
    \end{center}
        in which $R/\br{I:\br{x_{m-2}}}\cong A\br{P_m}$ fails the injectivity at $\lambda_m-1$ (Theorem \ref{thm_Pn}). Then $A\br{T_{m,3}}$ fails the injectivity (hence fails the WLP) at $\rho_m$. 
Conversely, one has $\lambda_m=\lambda_{m-1}=\lambda_{m+1}=\lambda_{m+2}$. In the case $m=17$, we can check directly by \texttt{Macaulay2} that $A\br{T_{17,3}}$ fails the injectivity at its mode minus 1. Consider $m\geq 15$ and $m\neq 17$. We have the exact sequence
        \begin{center}
            ${\xymatrix{
            R/I\ar@{->>}[r] & R/\br{I+\br{x_m}}\ar[r] & 0
            }}$
        \end{center}
         and $R/\br{I+\br{x_m}}\cong A\br{P_{m-1}}\otimes_\Bbbk A\br{P_3}$. Finally, for $m\geq 15$ and $m\neq 17$, $A\br{P_{m-1}}$ fails the surjectivity at $\lambda_{m-1}$, $A\br{P_3}$ fails the surjectivity at 0. Applying Lemma \ref{tensor}, $R/\br{I+\br{x_m}}$ fails the surjectivity at $\lambda_{m-1}+1=\rho_m+1$. So $A\br{T_{m,3}}$ fails the WLP in this case.\\
    \underline{\textbf{Case 3.} $I\br{T_{m,3};t}$ has the mode $\rho_m+2$.}\\
    \underline{\textbf{Subcase 3.1.} $\lambda_{m+2}\geq\rho_m+2$.} Then applying the same method as in Case 1.\\
    \underline{\textbf{Subcase 3.2.} $\lambda_m\leq\lambda_{m+2}\leq\rho_m+1\leq\lambda_{m-4}+2\leq\lambda_m+1$.}
    \begin{itemize}
        \item If $\lambda_m=\rho_m+1$, then $\lambda_m=\lambda_{m-1}+1$ (otherwise, we have $\lambda_m=\lambda_{m-1}\leq\rho_m=\lambda_m-1$, contracdition). Consider the exact sequence
        \begin{center}
            ${\xymatrix{
            0 \ar[r] & R/\br{I:\br{x_{m-2}}}(-1) \ar@{^{(}->}[r]^-{\cdot x_{m-2}} & R/I
            }}$
        \end{center}
        in which $R/\br{I:\br{x_{m-2}}}\cong A\br{P_m}$ fails the injectivity at $\lambda_m-1$ (Theorem \ref{thm_Pn}). Hence $A\br{T_{m,3}}$ fails the injectivity at $\rho_m+1$ (therefore fails the WLP).
        
        \item If $\lambda_m=\rho_m$, then consider the exact sequence 
        \begin{center}
            ${\xymatrix{
            0 \ar[r] & R/\br{I:\br{x_{m-2}}}(-1) \ar@{^{(}->}[r]^-{\cdot x_{m-2}} & R/I
            }}$
        \end{center}
        and $R/\br{I:\br{x_{m-2}}}\cong A\br{P_m}$. We will prove that the map $\times \ell :\brc{A\br{P_m}}_{\lambda_m}\rightarrow \brc{A\br{P_m}}_{\lambda_m+1}$ is not injective. With $m=16$, we have that the Hilbert series of $A\br{P_{16}}$ is
        \begin{center}
            $HS\br{A\br{P_{16}},t}=1+16t+105t^2+364t^3+715t^4+792t^5+462t^6+120t^7+9t^8.$
        \end{center}
        Then the above statement is true with $m=16$. For $m=15$ or $m\geq17$, $A\br{P_m}$ fails the surjectivity at $\lambda_m$. Note that $\dim_\Bbbk\brc{A\br{P_m}}_{\lambda_m}\geq\dim_\Bbbk\brc{A\br{P_m}}_{\lambda_m+1}$. Hence, the above map cannot be injective. 
        
        Therefore, the map $\times \ell:\brc{A\br{T_{m,3}}}_{\rho_m+1}$$\rightarrow\brc{A\br{T_{m,3}}}_{\rho_m+2}$ is not injective.
    \end{itemize}
\end{proof}

\begin{thm}\label{wlp_T4n}
    $A\br{T_{4,n}}$ has the WLP if and only if $n\in \brb{1,2,...,7,9,10,13}$.
\end{thm}
\begin{proof}
    By using \texttt{Macaulay2}, we can check that for $1\leq n\leq 17$, $A\br{T_{4,n}}$ has the WLP if and only if $n\in \brb{1,2,...,7,9,10,13}$. Consider $n\geq 18$ and see Figure \ref{fig4}.\\
    \underline{\textbf{Case 1.} $I\br{T_{4,n};t}$ has the mode $\lambda_{n+2}$.} Consider the exact sequence
    \begin{center}
        ${\xymatrix{
        R/I \ar@{->>}[r] & R/\br{I+\br{x_3}}\ar[r] & 0
        }}$
    \end{center}
    and $R/\br{I+\br{x_3}}\cong A\br{P_{n+3}}$ fails the surjectivity at degree $\lambda_{n+3}$ (Theorem \ref{thm_Pn}). Note that $\lambda_{n+3}\geq \lambda_{n+2}$. Then, $R/I$ fails the surjectivity at $\lambda_{n+3}$ (and hence, fails the WLP).\\
    \underline{\textbf{Case 2.} $I\br{T_{4,n};t}$ has the mode $\lambda_{n+2}+1$.}\\
    \underline{\textbf{Subcase 2.1.} $\lambda_{n+2}=\lambda_n$.} Consider the exact sequence
    \begin{center}
        ${\xymatrix{
        R/I\ar@{->>}[r] & R/\br{I+\br{x_4}}\ar[r] &0
        }}$
    \end{center}
    and $R/\br{I+\br{x_4}}\cong A\br{P_3}\otimes_\Bbbk A\br{P_n}$. Since $A\br{P_3}$ and $A\br{P_n}$ fail, respectively, the surjectivity at degree $0$ and $\lambda_n$, we have that $A\br{P_3}\otimes_\Bbbk A\br{P_n}$ fails the surjectivity at $0+\lambda_n+1$ $= \lambda_n+1$  by \Cref{tensor}, therefore so does $R/I$.\\    
    \underline{\textbf{Subcase 2.2.} $\lambda_{n+2}=\lambda_n+1$.}\\
    \underline{\textbf{Subcase 2.2.1} $\lambda_{n+1}=\lambda_n+1$.} Consider the exact sequence
    \begin{center}
        ${\xymatrix{
        0\ar[r] & R/\br{I:\br{x_2}}(-1) \ar@{^{(}->}[r]^-{\cdot x_2}&R/I
        }}$
    \end{center}
     and $R/\br{I:\br{x_2}}\cong A\br{P_{n+1}}$ fails the injectivity at degree $\lambda_{n+1}-1$ (Theorem \ref{thm_Pn}). Hence, $A\br{T_{4,n}}$ fails the injectivity at $\lambda_{n+2}$.\\
    \underline{\textbf{Subcase 2.2.2} $\lambda_{n+1}=\lambda_n$.} Then $\lambda_{n-1}=\lambda_n=\lambda_{n+1}$.
    \begin{itemize}
        \item If $\lambda_{n-1}=\lambda_{n-2}+1$, then consider the exact sequence
        \begin{center}
            ${\xymatrix{0\ar[r] & R/\br{I:\br{x_4}}(-1) \ar@{^{(}->}[r]^-{\cdot x_4} & R/I}}$
        \end{center}
        and $R/\br{I:\br{x_4}}\cong A\br{P_{n-1}}\otimes_\Bbbk\frac{\Bbbk\brc{z}}{\br{z^2}}$. Since $A\br{P_{n-1}}$ and $\frac{\Bbbk\brc{z}}{\br{z^2}}$ fail, respectively, the injectivity at degree $\lambda_{n-1}-1$ and $1$, we have $A\br{P_{n-1}}\otimes_\Bbbk\frac{\Bbbk\brc{z}}{\br{z^2}}$ fails the injectivity at $\br{\lambda_{n-1}-1}+1=\lambda_{n-1}$ (Lemma \ref{tensor}). Then $R/I$ fails the injectivity at $\lambda_{n+2}$.
        
        \item If $\lambda_{n-1}=\lambda_{n-2}$, then $\lambda_{n-2}=\lambda_{n-3}+1$. Consider the exact sequence
        \begin{center}
            ${\xymatrix{
            0\ar[r] & R/\br{I:\br{y_1}}(-1) \ar@{^{(}->}[r]^-{\cdot y_1} & R/I
            }}$
        \end{center}
        and $R/\br{I:\br{y_1}}\cong A\br{P_{n-2}}\otimes_\Bbbk A\br{P_3}$. Since $A\br{P_{n-2}}$ and $A\br{P_3}$ fail, respectively, the injectivity at degree $\lambda_{n-2}-1$ and $1$, we have that $A\br{P_{n-2}}\otimes_\Bbbk A\br{P_3}$ fails the injectivity at $\br{\lambda_{n-2}-1}+1=\lambda_{n-2}$ by \Cref{tensor}. Then $R/I$ fails the injectivity at $\lambda_{n+2}$.
    \end{itemize}
\end{proof}
\begin{thm}\label{wlp_T5n}
    $A\br{T_{5,n}}$ has the WLP if and only if $n\in \brb{1,2,3,5,6,9}$.
\end{thm}
\begin{proof}
    By using \texttt{Macaulay2}, we can check that for $1\leq n\leq 16$, $A\br{T_{5,n}}$ has the WLP if and only if $n\in \brb{1,2,3,5,6,9}$. Consider $n\geq 17$ and see Figure \ref{fig5}.\\
    \underline{\textbf{Case 1.} $I\br{T_{5,n};t}$ has the mode $\lambda_{n+2}$.} Consider the exact sequence 
    \begin{center}
        ${\xymatrix{
        R/I \ar@{->>}[r] & R/\br{I+\br{x_4}} \ar[r] & 0
        }}$
    \end{center}
    in which $R/\br{I+\br{x_4}}\cong A\br{P_{n+4}}$ fails the surjectivity at degree $\lambda_{n+4}$ by \Cref{thm_Pn}. Note that $\lambda_{n+4}\geq \lambda_{n+2}$. Therefore, $R/I$ fails the surjectivity at degree $\lambda_{n+4}\geq\lambda_{n+2}$, and hence, $R/I$ fails the WLP.\\
    \underline{\textbf{Case 2.} $I\br{T_{5,n};t}$ has the mode $\lambda_{n+2}+1$.} If $\lambda_{n+4}\geq\lambda_{n+2}+1$ , then using the same method as in Case 1.
    Consider $\lambda_{n+4}\leq \lambda_{n+2}$, then $\lambda_{n+2}=\lambda_{n+3}=\lambda_{n+4}$\\
    \underline{\textbf{Subcase 2.1.} $\lambda_{n+2}=\lambda_{n+1}+1$.} Consider the exact sequence 
    \begin{center}
        ${\xymatrix{
        0\ar[r] &R/\br{I:\br{x_3}}(-1) \ar@{^{(}->}[r]^-{\cdot x_3} & R/I
        }}$
    \end{center}
     in which $R/\br{I:\br{x_3}}\cong A\br{P_{n+2}}$ fails the injectivity at degree $\lambda_{n+2}-1$ by \Cref{thm_Pn}. Therefore, $R/I$ fails the injectivity at degree $\lambda_{n+2}$, and hence, $R/I$ fails the WLP.\\  
    \underline{\textbf{Subcase 2.2.} $\lambda_{n+2}=\lambda_{n+1}$.} In this case, $n\geq19$ (note that $\lambda_n= \ceil{\frac{5n+2-\sqrt{5n^2+20n+24}}{10}}$, see \cite{NT2024}, Proposition 3.1) and $\lambda_n=\lambda_{n-1}=\lambda_{n-2}= \lambda_{n+1}-1$. Consider the exact sequence
    \begin{center}
        ${\xymatrix{
        R/I \ar@{->>}[r] &R/\br{I+\br{y_2}}\ar[r] & 0
        }}$
    \end{center}
    and $R/\br{I+\br{y_2}}\cong A\br{P_{n-2}}\otimes_\Bbbk A\br{\mbox{Pan}_5}$. The Hilbert series of $A\br{\mbox{Pan}_5}$ is $1+6t+9t^2+3t^3$ and $n-2\geq17$. Therefore, $A\br{P_{n-2}}$ and $A\br{\mbox{Pan}_5}$ fail, respectively, the surjectivity at degree $\lambda_{n-2}$ and $1$. Hence, $A\br{P_{n-2}}\otimes_\Bbbk A\br{\mbox{Pan}_5}$ fails the surjectivity at degree $\lambda_{n-2}+1+1=\lambda_{n+2}+1$  by \Cref{tensor}. Then $R/I$ fails the surjectivity at degree $\lambda_{n+2}+1$.\\
    \underline{\textbf{Subcase 3.} $I\br{T_{5,n};t}$ has the mode $\lambda_{n+2}+2$.} Consider the exact sequence
    \begin{center}
        ${\xymatrix{
        0\ar[r] & R/\br{I:\br{x_3}}(-1) \ar@{^{(}->}[r]^-{\cdot x_3} & R/I
        }}$
    \end{center}
    and $R/\br{I:\br{x_3}}\cong A\br{P_{n+2}}$. It follows from \Cref{thm_Pn} that $A\br{P_{n+2}}$ fails the surjectivity at degree $\lambda_{n+2}$. Moreover, $\dim_\Bbbk \brc{A\br{P_{n+2}}}_{\lambda_{n+2}}\geq\dim_\Bbbk \brc{A\br{P_{n+2}}}_{\lambda_{n+2}+1}$. Therefore, the map $\times \ell :\brc{A\br{P_{n+2}}}_{\lambda_{n+2}}\rightarrow \brc{A\br{P_{n+2}}}_{\lambda_{n+2}+1}$ cannot be injective, and so does the map $\times \ell :\brc{R/I}_{\lambda_{n+2}+1}\rightarrow \brc{R/I}_{\lambda_{n+2}+2}$. Then $R/I$ fails the WLP at degree $\lambda_{n+2}+1$.
\end{proof}

\section{Acknowledgement}
This research is funded by Hue University of Education's Project under grant number T.24.TN.101.05.

\bibliographystyle{plain} 
\bibliography{WLP_Graphs} 

\begin{thebibliography}{10}

\bibitem{AB2020}
N.~Altafi and M.~Boij.
\newblock The weak {L}efschetz property of equigenerated monomial ideals.
\newblock {\em J. Algebra}, 556:136--168, 2020.

\bibitem{BMMNZ12}
M.~Boij, J.C. Migliore, R.M. Mir\'{o}-Roig, U.~Nagel, and F.~Zanello.
\newblock On the shape of a pure {$O$}-sequence.
\newblock {\em Mem. Amer. Math. Soc.}, 218(1024):viii+78, 2012.

\bibitem{DaoNair2022}
H.~Dao and R.~Nair.
\newblock On the lefschetz property for quotients by monomial ideals containing
  squares of variables.
\newblock {\em Communications in Algebra}, 52(3):1260--1270, 2024.

\bibitem{Macaulay2}
D.R. Grayson and M.E. Stillman.
\newblock Macaulay2, a software system for research in algebraic geometry.

\bibitem{GH83}
I.~Gutman and F.~Harary.
\newblock Generalizations of the matching polynomial.
\newblock {\em Utilitas Math.}, 24:97--106, 1983.

\bibitem{HMMNWW2013}
T.~Harima, T.~Maeno, H.~Morita, Y.~Numata, A.~Wachi, and J.~Watanabe.
\newblock {\em The {L}efschetz properties}, volume 2080 of {\em Lecture Notes
  in Mathematics}.
\newblock Springer, Heidelberg, 2013.

\bibitem{HL94}
C.~Hoede and X.L. Li.
\newblock Clique polynomials and independent set polynomials of graphs.
\newblock {\em Discrete Math.}, 125:219--228, 1994.

\bibitem{MNS2020}
J.~Migliore, U.~Nagel, and H.~Schenck.
\newblock The weak {L}efschetz property for quotients by quadratic monomials.
\newblock {\em Math. Scand.}, 126(1):41--60, 2020.

\bibitem{MMN2011}
J.C. Migliore, R.M. Mir\'{o}-Roig, and U.~Nagel.
\newblock Monomial ideals, almost complete intersections and the weak
  {L}efschetz property.
\newblock {\em Trans. Amer. Math. Soc.}, 363(1):229--257, 2011.

\bibitem{NT2024}
H.D. Nguyen and Q.H. Tran.
\newblock The weak lefschetz property of artinian algebras associated to paths
  and cycles.
\newblock {\em Acta Math Vietnam}, 49(3):523--544, 2024.

\bibitem{QHT2020}
Q.H. Tran.
\newblock The {L}efschetz properties of artinian monomial algebras associated
  to some graphs.
\newblock {\em Journal of Science, Hue University of Education}, 59(3):12--22,
  2021.

\end{thebibliography}
\end{document}